\documentclass[12pt]{amsart}
\usepackage{amssymb}
\usepackage{amsmath, amsfonts, amsthm, amssymb, cite}
\usepackage{color}
\usepackage[all]{xy}
\usepackage{fancyhdr}
\begin{document}



\newtheorem{theorem}{Theorem}[section]
\newcommand{\mar}[1]{{\marginpar{\textsf{#1}}}}
\numberwithin{equation}{section}
\newtheorem*{theorem*}{Theorem}
\newtheorem{prop}[theorem]{Proposition}
\newtheorem*{prop*}{Proposition}
\newtheorem{lemma}[theorem]{Lemma}
\newtheorem{corollary}[theorem]{Corollary}
\newtheorem*{conj*}{Conjecture}
\newtheorem*{corollary*}{Corollary}
\newtheorem{definition}[theorem]{Definition}
\newtheorem*{definition*}{Definition}
\newtheorem{remarks}[theorem]{Remarks}
\newtheorem*{remarks*}{remarks}
\newtheorem{rems}[theorem]{Remarks}
\newtheorem{rem}{Remark}
\newtheorem*{rem*}{Remark}
\newtheorem*{rems*}{Remarks}
\newtheorem{cor}[theorem]{Corollary}
\newtheorem{defn}[theorem]{Definition}
\newtheorem{thm}[theorem]{Theorem}

\newtheorem*{not*}{Notation}
\newcommand\pa{\partial}
\newcommand\cohom{\operatorname{H}}
\newcommand\Td{\operatorname{Td}}
\newcommand\Trig{\operatorname{Trig}}
\newcommand\Hom{\operatorname{Hom}}
\newcommand\End{\operatorname{End}}
\newcommand\Ker{\operatorname{Ker}}
\newcommand\Ind{\operatorname{Ind}}
\newcommand\cker{\operatorname{coker}}
\newcommand\oH{\operatorname{H}}
\newcommand\oK{\operatorname{K}}
\newcommand\codim{\operatorname{codim}}
\newcommand\Exp{\operatorname{Exp}}
\newcommand\CAP{{\mathcal AP}}
\newcommand\T{\mathbb T}
\newcommand{\M}{\mathcal{M}}
\newcommand\ep{\epsilon}
\newcommand\te{\tilde e}
\newcommand\Dd{{\mathcal D}}

\newcommand\what{\widehat}
\newcommand\wtit{\widetilde}
\newcommand\mfS{{\mathfrak S}}
\newcommand\cA{{\mathcal A}}
\newcommand\maA{{\mathcal A}}
\newcommand\maF{{\mathcal F}}
\newcommand\maN{{\mathcal N}}
\newcommand\cM{{\mathcal M}}
\newcommand\maE{{\mathcal E}}
\newcommand\cF{{\mathcal F}}
\newcommand\maG{{\mathcal G}}
\newcommand\cG{{\mathcal G}}
\newcommand\cH{{\mathcal H}}
\newcommand\maH{{\mathcal H}}
\renewcommand\H{{\mathcal H}}
\newcommand\cO{{\mathcal O}}
\newcommand\cR{{\mathcal R}}
\newcommand\cS{{\mathcal S}}
\newcommand\cU{{\mathcal U}}
\newcommand\cV{{\mathcal V}}
\newcommand\cX{{\mathcal X}}
\newcommand\cD{{\mathcal D}}
\newcommand\cnn{{\mathcal N}}
\newcommand\wD{\widetilde{D}}
\newcommand\wL{\widetilde{L}}
\newcommand\wM{\widetilde{M}}
\newcommand\wV{\widetilde{V}}
\newcommand\Ee{{\mathcal E}}
\newcommand{\npartial}{\slash\!\!\!\partial}
\newcommand{\Heis}{\operatorname{Heis}}
\newcommand{\Solv}{\operatorname{Solv}}
\newcommand{\Spin}{\operatorname{Spin}}
\newcommand{\SO}{\operatorname{SO}}
\newcommand{\ind}{\operatorname{ind}}
\newcommand{\Index}{\operatorname{Index}}
\newcommand{\ch}{\operatorname{ch}}
\newcommand{\rank}{\operatorname{rank}}
\newcommand{\G}{\Gamma}
\newcommand{\HK}{\operatorname{HK}}
\newcommand{\Dix}{\operatorname{Dix}}

\newcommand{\tM}{\tilde{M}}  
\newcommand{\tS}{\tilde{S}}
\newcommand{\tH}{\tilde{\mathcal H}}
\newcommand{\tg}{\tilde{g}}
\newcommand{\tx}{\tilde{x}}
\newcommand{\ty}{\tilde{y}}
\newcommand{\ox}{\otimes}


\newcommand{\abs}[1]{\lvert#1\rvert}
 \newcommand{\A}{{\mathcal A}}
        \newcommand{\D}{{\mathcal D}}
\newcommand{\HH}{{\mathcal H}}
        \newcommand{\LL}{{\mathcal L}}
        \newcommand{\B}{{\mathcal B}}
        \newcommand{\K}{{\mathcal K}}
\newcommand{\oo}{{\mathcal O}}
         \newcommand{\PP}{{\mathcal P}}
         \newcommand{\Q}{{\mathcal Q}}
        \newcommand{\s}{\sigma}
        \newcommand{\coker}{{\mbox coker}}
        \newcommand{\dd}{|\D|}
        \newcommand{\n}{\parallel}
\newcommand{\bma}{\left(\begin{array}{cc}}
\newcommand{\ema}{\end{array}\right)}

\newcommand{\bca}{\left(\begin{array}{c}}
\newcommand{\eca}{\end{array}\right)}
\newcommand{\sr}{\stackrel}
\newcommand{\da}{\downarrow}
\newcommand{\tD}{\tilde{\D}}
        \newcommand{\R}{\mathbb R}
        \newcommand{\C}{\mathbb C}
        \newcommand{\h}{\mathbb H}
\newcommand{\Z}{\mathcal Z}
\newcommand{\N}{\mathbb N}
\newcommand{\tto}{\longrightarrow}
\newcommand{\ZZ}{{\mathcal Z}}
\newcommand{\ben}{\begin{displaymath}}
        \newcommand{\een}{\end{displaymath}}
\newcommand{\be}{\begin{equation}}
\newcommand{\ee}{\end{equation}}
        \newcommand{\bean}{\begin{eqnarray*}}
        \newcommand{\eean}{\end{eqnarray*}}
\newcommand{\nno}{\nonumber\\}
\newcommand{\bea}{\begin{eqnarray}}
        \newcommand{\eea}{\end{eqnarray}}
\newcommand{\x}{\times}

\newcommand{\Ga}{\Gamma}
\newcommand{\e}{\epsilon}
\renewcommand{\L}{\mathcal{L}}
\newcommand{\supp}[1]{\operatorname{#1}}
\newcommand{\norm}[1]{\parallel\, #1\, \parallel}
\newcommand{\ip}[2]{\langle #1,#2\rangle}
\newcommand{\nc}{\newcommand}
\nc{\gf}[2]{\genfrac{}{}{0pt}{}{#1}{#2}}
\nc{\mb}[1]{{\mbox{$ #1 $}}}
\nc{\real}{{\mathbb R}}
\nc{\comp}{{\mathbb C}}
\nc{\ints}{{\mathbb Z}}
\nc{\Ltoo}{\mb{L^2({\mathbf H})}}
\nc{\rtoo}{\mb{{\mathbf R}^2}}
\nc{\slr}{{\mathbf {SL}}(2,\real)}
\nc{\slz}{{\mathbf {SL}}(2,\ints)}
\nc{\su}{{\mathbf {SU}}(1,1)}
\nc{\so}{{\mathbf {SO}}}
\nc{\hyp}{{\mathbb H}}
\nc{\disc}{{\mathbf D}}
\nc{\torus}{{\mathbb T}}
\newcommand{\tk}{\widetilde{K}}
\newcommand{\boe}{{\bf e}}\newcommand{\bt}{{\bf t}}
\newcommand{\vth}{\vartheta}
\newcommand{\CGh}{\widetilde{\CG}}
\newcommand{\db}{\overline{\partial}}
\newcommand{\tE}{\widetilde{E}}
\newcommand{\tr}{{\rm tr}}
\newcommand{\ta}{\widetilde{\alpha}}
\newcommand{\tb}{\widetilde{\beta}}
\newcommand{\txi}{\widetilde{\xi}}
\newcommand{\hV}{\hat{V}}
\newcommand{\IC}{\mathbf{C}}
\newcommand{\IZ}{\mathbf{Z}}
\newcommand{\IP}{\mathbf{P}}
\newcommand{\IR}{\mathbf{R}}
\newcommand{\IH}{\mathbf{H}}
\newcommand{\IG}{\mathbf{G}}
\newcommand{\IS}{\mathbf{S}}
\newcommand{\CC}{{\mathcal C}}
\newcommand{\CD}{{\mathcal D}}
\newcommand{\CS}{{\mathcal S}}
\newcommand{\CG}{{\mathcal G}}
\newcommand{\CL}{{\mathcal L}}
\newcommand{\CO}{{\mathcal O}}
\nc{\ca}{{\mathcal A}}
\nc{\cag}{{{\mathcal A}^\Gamma}}
\nc{\cg}{{\mathcal G}}
\nc{\chh}{{\mathcal H}}
\nc{\ck}{{\mathcal B}}
\nc{\cd}{{\mathcal D}}
\nc{\cl}{{\mathcal L}}
\nc{\cm}{{\mathcal M}}
\nc{\cn}{{\mathcal N}}
\nc{\cs}{{\mathcal S}}
\nc{\cz}{{\mathcal Z}}
\nc{\sind}{\sigma{\rm -ind}}

\newcommand\clFN{{\mathcal F_\tau(\mathcal N)}}       
\newcommand\clKN{{\mathcal K_\tau(\mathcal N)}}       
\newcommand\clQN{{\mathcal Q_\tau(\mathcal N)}}       %
\newcommand\tF{\tilde F}
\newcommand\clA{\mathcal A}
\newcommand\clH{\mathcal H}
\newcommand\clN{\mathcal N}
\newcommand\Del{\Delta}
\newcommand\g{\gamma}
\newcommand\eps{\varepsilon}
\newcommand\vf{\varphi}
\newcommand\E{\mathcal E}

\newcommand{\CDA}{\mathcal{C_D(A)}} 
\newcommand{\dslash}{{\pa\mkern-10mu/\,}}

\newcommand{\sepword}[1]{\quad\mbox{#1}\quad} 
\newcommand{\comment}[1]{\textsf{#1}}   

\nc{\nt}{\newtheorem}
\nc{\bra}{\langle}
\nc{\ket}{\rangle}
\nc{\cal}{\mathcal}
\nc{\frk}{\mathfrak}

\parindent=0.0in

 \title{Spectral flow for nonunital spectral triples}

\author{A. L. CAREY}
\address{Mathematical Sciences Institute,
 Australian National University\\
 Canberra ACT, 0200 AUSTRALIA, and\newline
 School of Mathematics and Applied Statistics\\
University of Wollongong\\
Wollongong NSW, 2500 AUSTRALIA\\
 e-mail: alan.carey@anu.edu.au\\}
\author{V. GAYRAL}
\address{Laboratoire de Math\'ematiques\\
Universit\'e Reims Champagne-Ardenne\\
Moulin de la Housse-BP 1039, 51687 Reims FRANCE\\
e-mail: victor.gayral@univ-reims.fr}
\author{J. PHILLIPS}
\address{Department of Mathematics and Statistics\\
University of Victoria\\
Victoria BC, CANADA\\
e-mail: johnphil@uvic.ca}
\author{A. RENNIE}
\address{School of Mathematics and Applied Statistics\\
University of Wollongong\\
Wollongong NSW, 2500 AUSTRALIA\\
e-mail: renniea@uow.edu.au}
\author{F. A. SUKOCHEV}
\address{School of Mathematics and Statistics,
University of New South Wales\\
Kensington NSW, 2052 AUSTRALIA\\
e-mail: f.sukochev@unsw.edu.au}

\begin{abstract}
We prove two results about nonunital index theory left open by \cite{CGRS2}. The
first is that the spectral triple arising from an action of the reals on a $C^*$-algebra with invariant trace
satisfies the hypotheses of the nonunital local index formula. The second result concerns the meaning of spectral flow in the nonunital case. For the special case of paths
arising from the odd
index pairing for smooth spectral triples in the nonunital setting we are able to connect with earlier approaches to the analytic definition of spectral flow. 
\end{abstract}
\maketitle

\tableofcontents
\parskip=4pt
\section{Introduction}
The local index formula in noncommutative geometry originated 
in the paper of Connes-Moscovici \cite{CM}.
Subsequent applications have revealed that it provides a unifying 
viewpoint for many formerly unrelated isolated classical theorems. 
It also produces a way  to calculate topological invariants for noncommutative algebras.

In \cite{CGRS2}, a local index formula (generalising both \cite{CM, Hi} and \cite{CPRS2,CPRS3}) 
was derived for nonunital spectral triples. Such spectral triples encompass as examples 
classical Dirac type operators on noncompact manifolds as well as noncommutative examples.
The local index formula of \cite{CGRS2} computes, in particular, a pairing of $K$-homology with 
$K$-theory using a generalisation of the residue cocycle first encountered in \cite{CM}. 
From a conceptual point of view,
this index pairing is defined using the Kasparov product.

Recall that a nonunital spectral triple $(\A, \H, \D)$ is given by 
a nonunital $*$-algebra $\A$ acting
on a Hilbert space $\H$, together with an unbounded self-adjoint operator $\D$
such that all commutators $[\D,a]$ are densely defined and
bounded, and $a(1+\D^2)^{-1/2}$ is compact for all $a\in\A$. Typically however,
$(1+\D^2)^{-1/2}$ is not compact.
In the odd case, it
was shown in \cite{CGRS2} that this $K$-theoretical
pairing can be realised as the index of a generalised Toeplitz operator
even in the nonunital setting. Whereas in the unital case the 
relationship between spectral flow and the Toeplitz theory is not 
difficult (see for example the discussion in \cite{BCPRSW})
a lengthier argument is needed in the nonunital case in order to 
explain the sense in which we are computing the spectral flow. 
The issue is that the residue formula appears to be using
a path of unbounded operators, none of which are Fredholm. This paper provides such an argument.

We present here two main results.
The first is that the index formula for generalised Toeplitz operators in \cite{PR},
arising from actions of the reals on a nonunital $C^*$-algebra, fits into the framework of
the nonunital local index formula of \cite{CGRS2}.

The second result justifies  the notion that the local index formula of
\cite{CGRS2} is computing spectral flow.  We follow an idea originating with I.M. Singer \cite{Singer},
 refined in \cite{G}, and introduce an exact one form on 
 a suitable affine space of  perturbations of $\D$. We then show how to write the index of the generalised Toeplitz operator of \cite{CGRS2} as the integral of this one form in a fashion which provides a direct comparison with the unital formula of \cite{CP2}. The idea is to reverse the argument in 
\cite{CPRS2} which goes from an integral formula for spectral flow to the resolvent cocycle formula. 
Thus we start from the resolvent cocycle in the nonunital setting
and derive from it
 a variant of the integral formulas for spectral flow that appear in \cite{CP1, CP2}. Our formula will apply to certain paths of operators with unitarily equivalent endpoints and is written in terms of paths of  operators that are possibly non-Fredholm. We remark that in the unital case this formula has had many applications and  its origins lie in the `variation of eta'  formula
that appears in Atiyah-Patodi-Singer \cite{APS3}.

The issue that arises in the nonunital case is that both bounded and unbounded Kasparov modules
(and thus spectral triples for nonunital algebras) do not lead directly to the study of Fredholm operators.
Rather one needs to modify the operator that appears in the definition of the Kasparov module in some fashion in order to obtain a Fredholm operator.  This fact is already well known in the traditional approach to Dirac type operators on non-compact manifolds where  one needs to twist the Dirac operator by special connections in order to have a Fredholm problem.
That this issue does have a sensible answer for the paths considered here suggests that
there may be broader classes of paths for which we can obtain spectral flow formulas, however we
 leave these speculative issues for the future.

The plan of the paper is as follows. In Section \ref{sec:cut-and-paste} 
we recall the integration and pseudodifferential operator 
theories (for nonunital spectral triples) of \cite{CGRS2}. In addition, Section \ref{sec:cut-and-paste} 
extends some results of \cite{CGRS2} to
identify an affine space of perturbations adapted to the above mentioned problem of spectral flow in the nonunital case.
All our constructions are done in the context of
general semifinite spectral triples, which is necessary to handle numerous examples, including
the generalised Toeplitz examples of \cite{PR}.

Section \ref{sec:PR} proves that there is a  (semifinite) spectral triple
that satisfies the hypotheses of the local index formula, such that the index theorems of Lesch, \cite{L},
and Phillips-Raeburn \cite{PR}, can be recovered using the procedure of \cite{CGRS2}. Indeed, 
the unital result of Lesch is already contained in \cite{CPRS2} (see also \cite{CPS2} for the 
connection to the spectral flow formula).

In the final Section \ref{sec:spec-flow}, we prove our
 main result. It states that given a spectral triple $(\A,\H,\D)$ satisfying the hypotheses that lead to the local index formula of \cite{CGRS2}, and a unitary $u\in \A^\sim$ in the minimal unitisation of $\A$, we can compute the
odd index pairing between $[u]\in K^1(\A)$ and $[(\A,\H,\D)]\in K_1(\A)$ using a formula analogous to those in \cite{CP1, CP2} for
spectral flow in the unital case. We stress that the path we consider here, namely $[0,1]\ni t\mapsto \D+tu[\D,u^*]$, need not be a
path of unbounded Fredholm operators. Nevertheless the method we adopt 
may be seen to determine, from our inital path, a related path of Fredholm operators and our formula
in terms of $\D$ computes
the spectral flow of this related Fredholm path. Moreover we show that this is also the index
of the generalised Toeplitz operator $PuP$ where $P$ is the non-negative spectral projection of $\D$
as would be expected given the formulations of \cite{BeF, CM} and \cite{BCPRSW} .





{\bf Acknowledgements}
AC, AR, FS acknowledge the support of the ARC, and JP acknowledges the support
of NSERC.  AC also acknowledges the Alexander von Humboldt Stiftung and thanks colleagues at the University of M\"unster for support while this research was undertaken.

\section{Technical preliminaries}
\label{sec:cut-and-paste}


\subsection{Background material}

In this preliminary section, we  import notation, definitions and results from \cite{CGRS2}.
In all that follows, $\D$ is a self-adjoint operator affiliated to a semifinite von Neumann algebra $\mathcal{N}$ equipped with faithful normal semifinite 
trace $\tau$, where $\mathcal{N}\subset \B(\H)$,
and $\H$ is a separable Hilbert space,

\begin{definition}\label{def:d-does-int}
For any positive number $s>0$, we define the weight $\vf_s$ on $\cn$ by
$$
T\in\cn_+\mapsto\vf_s(T):=\tau\big((1+\D^2)^{-s/4}T(1+\D^2)^{-s/4}\big)\in[0,+\infty].
$$
As usual, we set
$
\cn_{\vf_s}:={\rm span}\{\cn_{\vf_s,+}\}
={\rm span}\big\{\big(\cn_{\vf_s}^{1/2}\big)^*\cn_{\vf_s}^{1/2}\}\subset\cn,
$
where
$$
\cn_{\vf_s,+}
:=\left\{T\in\cn_+:\vf_s(T)<\infty\right\}\quad {\rm and} \quad
\cn_{\vf_s}^{1/2}:=\{T\in\cn:T^*T\in \cn_{\vf_s,+}\}.
$$
\end{definition}

With the notation as in Definition \ref{def:d-does-int},
the weights $\vf_s$, $s>0$, are faithful, normal and semifinite, \cite[Lemma 2.2]{CGRS2}.
We will also need the spaces $\mathcal{L}^p(\cn,\tau)$ of measurable operators $T$ affiliated to $\cn$ with
$\tau(|T|^p)<\infty$. With this notation, $\cn_\tau=\cn\cap\mathcal{L}^1(\cn,\tau)$ and 
$\cn_\tau^{1/2}=\cn\cap\mathcal{L}^2(\cn,\tau)$. This differs from the notation of \cite{CGRS2}.

\begin{definition} 
\label{def:define-all}
Retain the notation of Definition \ref{def:d-does-int}.

(i) For each $p\geq 1$ we define
$
\B_2(\D,p):=\bigcap_{s>p} \Big(\cn_{\vf_s}^{1/2}\bigcap \cn_{\vf_s}^{1/2*}\Big).
$\\
(ii) We set
$\B_1(\D,p)= \B_2(\D,p)^2:={\rm span}\{TS:\,T,\,S\in\B_2(\D,p)\}$.
%
%

(iii) Set $\HH_\infty=\bigcap_{k\geq 0} {\rm dom}\,\D^k$. For an 
operator $T\in\cn$ such that $T:\HH_\infty\to\HH_\infty$
we set
$$
\delta(T):=[(1+\D^2)^{1/2},T],\quad T\in \cn.
$$

(iv) In addition, if $T:\H_\infty\to\H_\infty$ we define $L(T),\,R(T):\H_\infty\to\H_\infty$ by

\begin{align}
\label{LR}
L(T):=(1+\mathcal{D}^2)^{-1/2}[\mathcal{D}^2,T],\quad R(T):=[\mathcal{D}^2,T](1+\mathcal{D}^2)^{-1/2}.
\end{align}
%
%
%
 
 (v) Define 
$
\B_q^k(\D,p)
:=\big\{T\in\cn\,:\,\forall l=0,\dots,k,
\, \delta^l(T)\in \B_q(\D,p)\big\}.
$ for $k=0,1,2,\dots,\infty$ and $q=1,2$.
%
\end{definition}
The spaces $\B_2(\D,p)$ and $\B_1(\D,p)$ are Fr\'echet subalgbras of $\mathcal N$ (see \cite{CGRS2} subsections 2.1 and 2.2). The natural
topology of $\B_2(\D,p)$ is determined
by the family of seminorms 
\begin{equation}\label{pn}
 \Q_n(T)
:=\left(\|T\|^2+\varphi_{p+1/n}(|T|^2)+\varphi_{p+1/n}(|T^*|^2)\right)^{1/2},\quad n=1,2,3\dots,
\end{equation}
and the topology of $\B_1(\D,p)$ is then determined by the family of seminorms 
\begin{equation}
\label{new-norm}
\PP_n(T)
:=\inf\Big\{\sum_{i=1}^k\Q_n( T_{1,i})\,\Q_n(T_{2,i})\ :\  
T=\sum_{i=1}^kT_{1,i}T_{2,i}\Big\},\quad n=1,2,3\dots.
\end{equation}
We equip $\B_1^k(\D,p)$ and $\B_2(\D,p)$, $k=0,1,2,\dots,\infty$, 
with the topology determined by the seminorms $\PP_{n,l}$ defined by
$$
 \PP_{n,l}(T)
:=\sum_{j=0}^l\PP_n(\delta^j(T)), \mbox{ and } \Q_{n,l}(T)
:=\sum_{j=0}^l\Q_n(\delta^j(T))
\quad n=1,2,\dots,\quad l=0,\dots,k.
$$
\begin{definition}
 \label{op0}
The set of {\bf regular order-$r$ pseudodifferential operators} is
$$
{{\rm OP}^r}(\D):=
(1+\D^2)^{r/2}\Big(\bigcap_{n\in\mathbb N}{\rm dom}\,\delta^n\Big),\quad r\in\mathbb R, 
\qquad {\rm OP}^*(\D):=\bigcup_{r\in\R}{\rm OP}^r(\D).
$$
The set of
{\bf  order-$r$ tame pseudodifferential operators} associated with 
$(\HH,\D)$ and $(\cn,\tau)$ for $p\geq 1$ is given by
$$
{\rm OP}^r_0(\D):=(1+\D^2)^{r/2}\B_1^\infty(\D,p) ,\quad r\in\mathbb R,\qquad {\rm OP}^*_0(\D)
:=\bigcup_{r\in\R}{\rm OP}^r_0(\D).
$$
We topologise  ${\rm OP}^r_0(\D)$  with the family of norms
$$
\PP_{n,l}^r(T):=\PP_{n,l}\big((1+\D^2)^{-r/2}T\big), \quad n,l\in\mathbb N.
$$
\end{definition}

With this definition, 
${\rm OP}^r(\D)$ and ${\rm OP}^r_0(\D)$ are Fr\'echet spaces, while ${\rm OP}^0(\D)$
and  ${\rm OP}^0_0(\D)$ are Fr\'echet $*$-algebras, and \cite[Lemma 1.31]{CGRS2} 
proves that ${\rm OP}^r(\D){\rm OP}^t_0(\D),\,{\rm OP}^t_0(\D){\rm OP}^r(\D)\subset {\rm OP}^{r+t}_0(\D)$.
In particular, $\B_1^\infty(\D,p)={\rm OP}^0_0(\D)$ is a two-sided $*$-ideal in 
${\rm OP}^0(\D)=\cap{\rm dom}\,\delta^m$. 
In \cite[Corollary 1.30]{CGRS2}  it is shown that 
$\bigcup_{r<-p}{\rm OP}^r_0(\D)\subset \L^1(\cn,\tau)\cap\cn=\cn_\tau$, which is the basic justification
for the introduction  of tame pseudodifferential operators in the nonunital setting.


The last ingredient from the pseudodifferential calculus is the complex one parameter group
of automorphisms 
 on ${\rm OP}^*(\D)$, defined by
\begin{equation}
\label{sigma}
\sigma^z(T):=(1+\D^2)^{z/2}\,T\,(1+\D^2)^{-z/2},\quad z\in\C,\ T\in {\rm OP}^*(\D).
\end{equation}
This group is strongly continuous and preserves each of the spaces
${\rm OP}^r(\D)$ and ${\rm OP}^r_0(\D)$, $r\in\R$ (see \cite{CGRS2} subsection 2.4).

Next we recall the definition of spectral triple, and summability of spectral triples, from \cite{CGRS2}.
\begin{definition} 
\label{def:ST}
A  semifinite
spectral triple $(\A,\HH,\D)$, relative to $(\cn,\tau)$, is given by a Hilbert space $\HH$, a
$*$-subalgebra  $\A\subset\, \cn$
acting on
$\HH$, and a densely defined unbounded self-adjoint operator $\D$ affiliated
to $\cn$ such that:

0. For all $a\in\A$, $a:\,{\rm dom}\D\to{\rm dom}\D$;

1. $da:=[\D,a]$ is densely defined and extends to a bounded operator in
$\cn$ for all $a\in\A$;

2. $a(1+\D^2)^{-1/2}\in\K(\cn,\tau)$ for all $a\in\A$, where $\K(\cn,\tau)$ is the ideal of $\tau$ compact 
operators in $\cn$ (the norm closure of the algebra generated by finite trace projections).

We say that $(\A,\HH,\D)$ is even if in addition there is a $\mathbb{Z}_2$-grading
such that $\A$ is even and $\D$ is odd. This means there is an operator $\gamma$ such that
$\gamma=\gamma^*$, $\gamma^2={\rm Id}_\cn  
$, $\gamma a=a\gamma$ for all $a\in\A$ and
$\D\gamma+\gamma\D=0$. Otherwise we say that $(\A,\HH,\D)$ is odd.
%
%

A  semifinite spectral triple  $(\A,\cH,\D)$, is said to be finitely summable if
there exists $s>0$ such that for all $a\in\A$, $a(1+\D^2)^{-s/2}\in\cl^1(\cn,\tau)$.
In such a case, we  let
$$
p:=\inf\big\{s>0\,:\,\forall a\in\A, \,\, \tau\big(|a|^{1/2}(1+\D^2)^{-s/2}|a|^{1/2}\big)<\infty
\big\},
$$
and call $p$ the spectral dimension of $(\A,\HH,\D)$.
\end{definition}

It is shown in \cite[Propositions 3.16,\,3.17]{CGRS2} that $\A\subset \B_1(\D,p)$ is a necessary condition
for $(\A,\H,\D)$ to be finitely summable with spectral dimension $p$, and that this condition
is almost sufficient as well.
 
\begin{definition}
\label{delta-phi}
Let $(\A,\HH,\D)$ be a  semifinite spectral triple relative to $(\cn,\tau)$. Then we say that $(\A,\HH,\D)$ is
smoothly summable if 
$
\A\cup[\D,\A]\subset \B_1^\infty(\D,p).
$
\end{definition}

\subsection{An affine space of perturbations}

This subsection proves that the self-adjoint part of $\B_1^\infty(\D,p)$ provides an 
affine space of perturbations of an operator $\D$ suitable for the purpose of studying spectral flow as an integral of a one form. We begin with some preliminary lemmas.

\begin{lemma}\label{V simplification Z} 
For $B\in{\rm OP}^0(\D)_{\rm sa}$, set $\D_B:=\D+B$. 
Then  $(1+\mathcal{D}_B^2)^{s/2}$ belongs to ${\rm OP}^{s}(\D)$ 
for every $s\in\mathbb{R}$.
\end{lemma}
\begin{proof}
Clearly, $1+\D_B^2\in{\rm OP}^2(\D)$. So by \cite[Proposition 2.30]{CGRS2},
  $(1+\D_B^2)(1+\D^2)^{-1}$
and $(1+\D^2)^{-1}(1+\D_B^2)$ belong to
${\rm OP}^0(\D)$. Next, we prove that $(1+\D_B^2)^{-1}\in{\rm OP}^{-2}(\D)$, which 
is equivalent to $(1+\D_B^2)^{-1}(1+\D^2)\in{\rm OP}^0(\D)$. But we already know, 
by writing $\D=\D_B-B$, that
$(1+\D_B^2)^{-1}(1+\D^2)\in{\rm OP}^0(\D_B)\subset\cn$, so that $(1+\D_B^2)^{-1}(1+\D^2)$ is
bounded.
It remains to show that $\delta^k\big((1+\D_B^2)^{-1}(1+\D^2)\big)\in\cn$, for all $k=1,2,\dots$.
For $k=1$, we have
$$
\delta\big((1+\D_B^2)^{-1}(1+\D^2)\big)=-(1+\D_B^2)^{-1}(1+\D^2)(1+\D^2)^{-1}
\delta(\D_B^2)(1+\D_B^2)^{-1}(1+\D^2),
$$
which is bounded as $(1+\D_B^2)^{-1}(1+\D^2)$ is bounded and 
$(1+\D^2)^{-1}\delta(\D_B^2)\in{\rm OP}^0(\D)$. An easy inductive argument shows that
$\delta^k\big((1+\D_B^2)^{-1}(1+\D^2)\big)$ is bounded 
for every $k\in\N$. Taking products, we deduce from the cases 
$s=\pm1$ that $(1+\mathcal{D}_B^2)^n\in
{\rm OP}^{2n}(\D)$ for every $n\in\mathbb{Z}$. Take now an arbitrary $s\in\mathbb{R}$ and
write $s=n-\alpha$ with $n\in\mathbb Z$ and $\alpha\in(0,1)$. Thus, it remains 
to show that for such $\alpha$,  $(1+\mathcal{D}_B^2)^{-\alpha}$ belongs to
${\rm OP}^{-2\alpha}(\D)$. 
For this, we use the integral formula for fractional powers
$$
(1+\D_B^2)^{-\alpha}
={(\sin(\pi\alpha))}^{-1}\int_0^\infty \lambda^{-\alpha}
(1+\lambda+\D_B^2)^{-1}d\lambda.
$$
Writing 
$
(1+\lambda+\D_B^2)^{-1}=(1+\lambda+\D^2)^{-1}-(1+\lambda+\D^2)^{-1}\big(
\D B+B\D_B\big)(1+\lambda+\D_B^2)^{-1},
$
we arrive at
$$
(1+\D^2)^\alpha(1+\D_B^2)^{-\alpha}={\rm Id}_\cn\\
-\frac{1}{\sin(\pi\alpha)}\int_0^\infty \lambda^{-\alpha}
(1+\D^2)^\alpha(1+\lambda+\D^2)^{-1}\big(\D B+B\D_B\big)(1+\lambda+\D_B^2)^{-1} d\lambda.$$
We estimate the integrand in operator norm using
\begin{align*}
\big\|(1+\D^2)^\alpha(1+\lambda+\D^2)^{-1}
\D B(1+\lambda+\D_B^2)^{-1}\big\|\leq \|B\| (1+\lambda)^{-3/2+\alpha}\\
\big\|(1+\D^2)^\alpha(1+\lambda+\D^2)^{-1}
B\D_B(1+\lambda+\D_B^2)^{-1}\big\|\leq \|B\| (1+\lambda)^{-3/2+\alpha},
\end{align*}
showing the norm-convergence of the integral. Writing next,
\begin{align*}
&\delta\big((1+\D_B^2)^{-\alpha}\big)
=\frac{1}{\sin(\pi\alpha)}\int_0^\infty \lambda^{-\alpha}
\delta\big((1+\lambda+\D_B^2)^{-1}\big)d\lambda\\
&= -\frac{1}{\sin(\pi\alpha)}\int_0^\infty \lambda^{-\alpha}
(1+\lambda+\D_B^2)^{-1}\big(\delta(B)\D_B+\D_B\delta(B)\big)
(1+\lambda+\D_B^2)^{-1}d\lambda,
\end{align*}
we obtain the estimate
\begin{align*}
&\big\|(1+\D^2)^\alpha\delta\big((1+\D_B^2)^{-\alpha}\big)\big\|\\
&
\leq C_{\alpha,y}\int_0^\infty \lambda^{-\alpha}
\big\|(1+\D^2)^\alpha(1+\lambda+\D_B^2)^{-1}\big(\delta(B)\D_B+\D_B\delta(B)\big)
(1+\lambda+\D_B^2)^{-1}\big\|d\lambda\\
&\leq 2 C_{\alpha,y} \|\delta(B)\|\int_0^\infty
\lambda^{-\alpha} (1+\lambda)^{-3/2+\alpha}d\lambda,
\end{align*}
which converges since $\alpha\in(0,1)$. On the basis of this, an easy recursive argument 
shows that $(1+\D^2)^\alpha\delta^k\big((1+\D_B^2)^{-\alpha}\big)$ is bounded for any 
$k\in\N$. This completes the proof.
\end{proof}

We then deduce an immediate corollary.
\begin{cor}
\label{cor:bd}
Let $B\in{\rm OP}^0(\D)_{\rm sa}$. 
Then  $(1+\mathcal{D}_B^2)^s(1+\mathcal{D}^2)^{-s}$ is bounded  
for every $s\in\mathbb{R}$.
\end{cor}

We have next our first preliminary result concerning affine spaces of perturbations.

\begin{prop}\label{affine1}
Let $B\in{\rm OP}^0(\D)_{\rm sa}$ and $p\geq 1$. Then 
we have $\mathcal{B}_2(\mathcal{D}_B,p)=\mathcal{B}_2(\mathcal{D},p)$ (resp.
$\mathcal{B}_1(\mathcal{D}_B,p)=\mathcal{B}_1(\mathcal{D},p)$) with equivalent 
$\mathcal{Q}_n$-seminorms (resp. $\PP_n$-seminorms). In particular, $\E_2:=\D+\B_2(\D,p)$ 
(resp. $\E_1:=\D+\B_1(\D,p)$) is an affine  sub-space of ${\rm OP}^1(\D)$,
whose Fr\'{e}chet topology is independent of the base-point.
\end{prop}
\begin{proof} 
Let $T\in\cn_+$ and $s>0$. We have by Corollary \ref{cor:bd} 
\begin{align*}
&\tau\big((1+\D_B^2)^{-s/4}T(1+\D^2_B)^{-s/4}\big)\\
&\qquad=\tau\big((1+\D_B^2)^{-s/4}(1+\D^2)^{s/4}(1+\D^2)^{-s/4}T(1+\D^2)^{-s/4}(1+\D^2)^{s/4}(1+\D^2_B)^{-s/4}\big)\\
&\qquad\leq \Vert (1+\D^2)^{-s/4}(1+\D^2_B)^{-s/4}\Vert^2\ \tau\big((1+\D^2)^{-s/4}T(1+\D^2)^{-s/4}\big).
\end{align*}
Similarly, we obtain
$$
\tau\big((1+\D^2)^{-s/4}T(1+\D^2)^{-s/4}\big)\leq 
\Vert (1+\D_B^2)^{-s/4}(1+\D^2)^{-s/4}\Vert^2\
\big((1+\D_B^2)^{-s/4}T(1+\D^2_B)^{-s/4}\big).
$$
Thus, the weights $\vf_s$ defined with $\D$ or with $\D_B$ are equivalent.
Substituting $s=p+4/n$ and comparing with the definition of the norms $\Q_n$ 
and $\PP_n$ completes the proof.
\end{proof}

To state an analogous result in the smooth case, namely when we use $\B_2^\infty$ and $\B_1^\infty$, 
we will compare 
the operators $L$ given in \eqref{LR}, associated with $\D$ and $\D_B$.
We arrive now at the main technical result.

\begin{prop}
\label{affine3}
Let $B\in {\rm OP}^0(\D)_{\rm sa}$. Then $ {\rm OP}^0(\D_B)= {\rm OP}^0(\D)$ and 
$ \B^\infty_1(\D_B,p)=:{\rm OP}^0_0(\D_B)= {\rm OP}^0_0(\D):=\B_1^\infty(\D,p)$ with equivalent topologies.  
In particular, $\D+{\rm OP}^0(\D)$  is an affine Fr\'{e}chet sub-space of ${\rm OP}^1(\D)$, 
whose topology is independent of the base-point.
\end{prop}
\begin{proof}
We need first to prove that $\cap_{k\in\mathbb N}{\rm dom}\,\delta^k=\cap_{k\in\mathbb N}{\rm dom}\,
\delta_B^k$ where $\delta_B(\cdot):=[(1+\D_B^2)^{1/2},\cdot]$.  Using $\cap_{k\in\mathbb N}{\rm dom}\,\delta^k=\cap_{k\in\mathbb N}{\rm dom}\,L^k$, see \cite{CPRS2} for a proof, 
we see that we equivalently  need to prove that 
$\cap_{k\in\mathbb N}{\rm dom}\,L^k=\cap_{k\in\mathbb N}{\rm dom}\,L_B^k$,
where $L_B$ is the linear operator defined in \eqref{LR} with $\D_B$ instead of $\D$.
For this, we observe the relation
\begin{align*}
\D_B^2-\D^2=\D B+B\D+B^2=(1+\D^2)^{1/2}( \D(1+\D^2)^{-1/2} B+\sigma^{-1}(B)
\D{(1+\D^2)^{-1/2}})+B^2,
\end{align*}
where $\sigma$ is the one-parameter complex group of automorphisms (for $\D$) given in \eqref{sigma}.
Defining the transformation 
$
T:B\mapsto  \D{(1+\D^2)^{-1/2}} B+\sigma^{-1}(B)
\D{(1+\D^2)^{-1/2}}+B^2,
$
we see that $T$ maps ${\rm OP}^0(\D)$ to itself (and similarly for ${\rm OP}^0_0(\D)=
\B_1^\infty(\D,p)$). 
Moreover,
we have the following relations between the maps $L$ and $L_B$:
\begin{align}
\label{LLB}
L_B(\cdot)&=(1+\D^2_B)^{-1/2}(1+\D^2)^{1/2}\big(L(\cdot)
+[T(B),\cdot]\big)+(1+\D^2_B)^{-1/2}\delta(\cdot) T(B),
\nonumber\\
L(\cdot)&=(1+\D^2)^{-1/2}(1+\D^2_B)^{1/2}\big(L_B(\cdot)-[T(B),\cdot]\big)-(1+\D^2)^{-1/2}\delta_B(\cdot) T(B),
\end{align}
By Lemma \ref{V simplification Z}, we have that $(1+\D^2_B)^{-1/2}(1+\D^2)^{1/2}$ 
and $(1+\D^2)^{-1/2}(1+\D^2_B)^{1/2}$ belong to ${\rm OP}^0(\D)$ and by the replacement
$(\D,B)\mapsto(\D_B,-B)$, they also belong to ${\rm OP}^0(\D_B)$. 
Now, the first equation of \ref{LLB} shows that $B$ belongs to the domain of $L_B$. 
By an iterative use of this equation, we deduce that $B\in \cap_{k\in\mathbb N}{\rm dom}\,L_B^k
={\rm OP}^0(\D_B)$. Last, writing $\D=\D_B-B$, Lemma
\ref{V simplification Z} applied to $\D_B$, shows that  $T(B)$ also belongs to  ${\rm OP}^0(\D_B)$.
This is clearly enough
to conclude that $\cap_{k\in\mathbb N}{\rm dom}\,L^k=\cap_{k\in\mathbb N}{\rm dom}\,L_B^k$.
That $\B_1^\infty(\D_B,p)=\B_1^\infty(\D,p)$ now follows since $\B_1^\infty(\D,p)$ is an
ideal in ${\rm OP}^0(\D)$ and that all the transformations used above also preserve 
$\B_1^\infty(\D,p)$. 
\end{proof}

\begin{cor}
Let $B\in {\rm OP}^0(\D)_{\rm sa}$ and $r\in\R$. Then $ {\rm OP}^r(\D_B)= {\rm OP}^r(\D)$ and
$ {\rm OP}^r_0(\D_B)= {\rm OP}^r_0(\D)$.
\end{cor}
\begin{proof}
By definition and Proposition \ref{affine3},
$$
{\rm OP}^r(\D_B)={\rm OP}^0(\D_B)(1+\D_B^2)^{-r/2}={\rm OP}^0(\D)(1+\D_B^2)^{-r/2}
={\rm OP}^r(\D)(1+\D^2)^{r/2}(1+\D_B^2)^{-r/2}.
$$
By Lemma  \ref{V simplification Z}, $(1+\D^2)^{r/2}(1+\D_B^2)^{-r/2}\in{\rm OP}^0(\D)$,
thus, 
$$
{\rm OP}^r(\D_B)\subset {\rm OP}^r(\D)\cdot{\rm OP}^0(\D)\subset {\rm OP}^r(\D).
$$
Reversing the role of $(\D,B)$ and $(\D_B,-B)$, we get the second inclusion. The statements
about ${\rm OP}^r_0(\D)$ are proved the same way.
\end{proof}

We require one more technical estimate for later use.

\begin{lemma}
\label{lem4}
Let $\D$ be an unbounded self-adjoint operator affiliated 
with a von Neumann algebra $\cn$ and let $B\in{\rm OP}^0(\D)$.
Then for any numbers  $\rho>0$ and  $s\geq2\|B\|$, the operator 
$\big(1+\D^2\big)^\rho(1+\D^2+sB+s^2\big)^{-\rho}$ is bounded with
$
\sup_{s\geq 2\|B\|}\big\|\big(1+\D^2\big)^\rho\big(1+\D^2+sB+s^2\big)^{-\rho}\big\|<\infty\,.
$
\end{lemma}

\begin{proof}
We let $C_\rho=(1+\D^2)^\rho(1+\D^2+sB+s^2)^{-\rho}$. For $\rho=0$, $C_\rho$ is bounded, and 
also for $\rho=1$, we have 
$$
\big(1+\D^2\big)\big(1+\D^2+sB+s^2\big)^{-1}=1-(sB+s^2)\big(1+\D^2+sB+s^2\big)^{-1}\,,
$$
and thus for $s\geq2\|B\|$, we obtain
$$
\Vert C_1\Vert=\big\|\big(1+\D^2\big)\big(1+\D^2+sB+s^2\big)^{-1}\big\|\leq1+(s\|B\|+s^2
) (1+s^2/2)^{-1}\leq 4\,.
$$
For $0<\rho<1$ we observe that $C_1$ is invertible, and so there is some positive constant $b=b(s)$ such
that
$$
0<b\leq (1+\D^2)(1+\D^2+sB+s^2)^{-2}(1+\D^2)\leq 16.
$$
Conjugating by $(1+\D^2)^{-1}$ and raising to the power $\rho$ yields, by operator monotonicity,
$$
b^{\rho}(1+\D^2)^{-2\rho}\leq (1+\D^2+sB+s^2)^{-2\rho}\leq (16)^\rho(1+\D^2)^{-2\rho},
$$
and conjugating by $(1+\D^2)^\rho$ yields
$$
b^{\rho}\leq (1+\D^2)^{\rho}(1+\D^2+sB+s^2)^{-2\rho}(1+\D^2)^{\rho}\leq (16)^\rho.
$$
Hence $\Vert C_\rho\Vert\leq 4^\rho$ independent of $s$ whenever $0\leq \rho\leq1$.

So, let us assume that the result hold for some given $\rho$. Then for $C_{\rho+1}$ we find
\begin{align*}
C_{\rho+1}&=\big(1+\D^2\big)^{\rho+1}\big(1+\D^2+sB+s^2\big)^{-\rho-1}\\
&=\big(1+\D^2\big)\big(1+\D^2+sB+s^2\big)^{-1}
\big(1+\D^2\big)^{\rho}\big(1+\D^2+sB+s^2\big)^{-\rho}\\
&\qquad\qquad\qquad
+\big(1+\D^2\big)\Big[\big(1+\D^2\big)^\rho,(1+\D^2+sB+s^2\big)^{-1}\Big]\big(1+\D^2+sB+s^2\big)^{-\rho}\\
&=C_1C_\rho-s(1+\D^2)(1+\D^2+sB+s^2)^{-1}[(1+\D^2)^\rho,B](1+\D^2+sB+s^2)^{-\rho-1}\\
&=C_1C_\rho-C_1(\sigma^{2\rho}(B)-B)(1+\D^2)^{\rho}(1+\D^2+sB+s^2)^{-\rho}\,s(1+\D^2+sB+s^2)^{-1}\\
&=C_1C_\rho-C_1(\sigma^{2\rho}(B)-B)C_\rho\,s(1+\D^2+sB+s^2)^{-1}.
\end{align*}
Now it is straightforward to show that $\Vert s(1+\D^2+sB+s^2)^{-1}\Vert$ is bounded 
independently of $s\geq 2\Vert B\Vert$, so if $C_\rho$ is bounded uniformly in $s$, so too is $C_{\rho +1}$. 
This completes the proof.
\end{proof}

\section{Nonunital Phillips-Raeburn examples}
\label{sec:PR}

In this Section we prove that the examples studied by Phillips and Raeburn in 
\cite{PR} give rise to smoothly summable semifinite spectral triples. 
We begin by recalling the construction in \cite{PR}
 in order to set our notation and assumptions.
 To this end,
$A$ will denote a $C^*$-algebra (usually non-unital) with a fixed faithful, 
norm-lower semi-continuous, densely defined trace, $\tau$, 
which is invariant under a strongly continuous, isometric action of the reals, $\alpha: \R\to {\rm Aut}(A)$. 
We let $A_\tau$ denote the dense ideal of trace-class elements in the $C^*$-algebra $A$, that is
$$
A_\tau=\{a\in A\,|\,\tau(|a|)<\infty\}={\rm span}\{a\in A_+\,|\,\tau(a)<\infty\}.
$$ 
We define a Banach-$*$-algebra norm on $A_\tau$ via $\|a\|_\tau = \|a\| + \tau(|a|):=\|a\|+\|a\|_1$, 
and observe that the action $\alpha$ restricts to a strongly
continuous action of $\R$ as isometric $*$-automorphisms of $A_\tau$.

Now $\alpha$ determines densely defined derivations, $\partial$ and $\partial_\tau$ on $A$ and $A_\tau$ respectively, 
given by the formulas
$$
\partial(a)=\lim_{t\to 0}\frac{\alpha_t(a) - a}{t} \;\;a\in A\;\;{\rm and}\;\; 
\partial_\tau(a)=\lim_{t\to 0}\frac{\alpha_t(a) - a}{t} \;\;a\in A_\tau,
$$
where the limit in each case is taken with respect to the complete norm topologies of the 
respective algebras. Moreover, ${\rm dom}(\partial_\tau)\subseteq
{\rm dom}(\partial)$ and $\partial\vert_{{\rm dom}(\partial_\tau)}=\partial_\tau$.

\begin{prop}
The smooth $*$-subalgebra, $\cap_{k=1}^\infty {\rm dom}(\partial^k)$ is dense in $A$ and the smooth $*$-subalgebra,
$\cap_{k=1}^\infty {\rm dom}(\partial_\tau^k)$ is dense in $A_\tau$.
\end{prop}

\begin{proof}
Let $f$ be a smooth compactly supported complex valued function on $\R$. 
If $a\in A$ define $a_f=\int_{-\infty}^\infty f(t)\alpha_t(a) dt \in A$.
By a change of variable we get
$$
\partial(a_f)=\lim_{s\to 0}\frac{\alpha_s(a_f)-a_f}{s}=\lim_{s\to 0}-\int_{\R}\frac{f(r-s)-f(r)}{-s}\alpha_r(a)dr=-\int_{\R} f^\prime(r)\alpha_r(a) dr= -a_{f^\prime}.
$$ 
By induction, we have $\partial^k(a_f)=(-1)^k a_{f^{(k)}}$ where $f^{(k)}$ 
denotes the $k$-th derivative of $f$. Hence, $a_f\in\cap_{k=1}^\infty {\rm dom}(\partial^k).$
Now take a sequence $\{f_n\}$ of non-negative smooth bump functions symmetric 
about $0,$ each with integral $1,$ and supports shrinking to $\{0\}$.
Then
\begin{align*}
\Vert a_{f_n}-a\Vert &=\Vert \int_\R f_n(t)\alpha_t(a)dt-a\Vert
=\Vert \int_\R f_n(t)\alpha_t(a)dt-\int_\R f_n(t)adt\Vert L^2({\R},H_\tau)\\
&=\Vert \int_\R f_n(t)(\alpha_t(a)-a)dt\Vert
\leq \sup_{t\in{\rm supp}(f_n)}\Vert \alpha_t(a)-a\Vert
\end{align*}
and we see that $\|a_{f_n} -a\|\to 0$ as $n\to \infty$ by the strong continuity of $\alpha$. 
The same argument works equally well with $A_\tau.$
\end{proof}

\subsection{The induced representation of the crossed product of $A$ by $\R$}

In this subsection we review some well known facts about crossed products in order to set notation and to recall the framework of \cite{PR}.
The first thing to recall here is that $\R$ is amenable, so that there is no distinction 
between the full and reduced crossed products. We denote the crossed product
by $A\rtimes_\alpha\R$.
We remind the reader of the multiplication  and involution for  
$x,\,y\in L^1({\R},A)\subset A\rtimes_\alpha{\R}$:
$$
(x*_\alpha y)(r) = \int_{\R}x(t)\alpha_t(y(r-t))dt,\qquad x^*(r)=\alpha_r(x(-r))^*.
$$
We let $H_\tau=L^2(A,\tau)$ be the (GNS) Hilbert space completion of the pre-Hilbert space
$A_\tau^{1/2}:=\{a\in A\,|\, a^*a\in A_\tau\}.$ Of course the action of 
$A$ on the ideal $A_\tau^{1/2}$ by left multiplication extends to a $*$-representation
of $A$ on $H_\tau.$ We denote this $*$-representation by juxtaposition 
since if $a\in A$ and $b\in A_\tau^{1/2}$, then the action of $a$ on the vector $b$ is just $ab.$
\begin{definition}
The covariant pair $(\pi,\lambda)$ of representations of $(A,{\R})$ on $L^2({\R},H_\tau)\cong
L^2({\R})\otimes H_\tau$ is defined by 
taking for  $\xi\in L^2({\R},H_\tau)$, $a\in A$ and $t,\,s\in{\R}$ 
$$
(\pi(a)\xi)(s):=\alpha_s^{-1}(a)\xi(s)\;\;{\rm and}\;\;(\lambda(t)\xi)(s):=\xi(s-t).
$$ 
\end{definition}
Then one easily checks the covariance condition $\lambda(t)\pi(a)\lambda(-t)=\pi(\alpha_t(a))$.
Thus, we get a $*$-representation $\tilde\pi$ of the crossed product algebra $A\rtimes_\alpha{\R}$ 
on the Hilbert space $L^2({\R},H_\tau)$, which for $x$ in the algebra $L^1({\R},A)\subset A\rtimes_\alpha{\R}$ 
and $\xi \in L^2({\R},H_\tau)$ is given by
$$
(\tilde\pi(x)\xi)(s)=\int_{\R}\alpha_s^{-1}(x(t))\xi(s-t)dt.
$$ 
One  checks directly that $\tilde\pi(x*_\alpha y)=\tilde\pi(x) \tilde\pi(y)$ as required.

Our interest now is in  
 $\mathcal{N}=(\tilde\pi(A\rtimes_\alpha{\R}))^{\prime\prime}$, 
 the von Neumann algebra generated by this representation. 
The essential point is that  $\mathfrak A = L^1({\R},A_\tau)\cap L^2({\R},H_\tau)$ 
is a Hilbert algebra with Hilbert space completion $L^2({\R},H_\tau)$ 
satisfying $\tilde\pi(\mathfrak A)^{\prime\prime} =\mathcal N.$ 
Moreover, letting
$M=\pi(A)''\subset \B(H_\tau)$ and  $\bar\tau$ be the normal extension of $\tau$ to $M$, then
with $M_{\bar\tau}$ the domain of $\bar\tau$ in $M$, and $M_{\bar\tau}^{1/2}$ the half-domain, 
we have $\pi(A)\cap M_{\bar\tau}=\pi(A_\tau)$ and 
$\pi(A)\cap M_{\bar\tau}^{1/2}=\pi(A^{1/2}_\tau)$. 
We also note that $H_\tau\cong\mathcal{L}^2(M,\bar\tau)$ is the GNS space 
of $M$ for $\bar\tau$ and thus $H_\tau\cap M=M_{\bar\tau}^{1/2}$. 

By 
Th\'{e}or\`{e}me 1, page 85 of \cite{Dix} there is an induced
faithful, normal, semifinite trace $\hat\tau$ on $\mathcal N$ 
which for products of elements in $x,\,y$ in $L^2({\R},H_\tau)$ such that 
$\tilde\pi(x),\tilde\pi(y)\in\cn$,
 is defined by
\begin{equation}
\hat\tau(\tilde\pi(x)^*\tilde\pi(y)):=\bra x | y\ket=\int_{\R}\bar \tau(x(t)^*y(t))dt.
\label{eq:trace-formula}
\end{equation}


\subsection{Constructing a nonunital spectral triple}

We have already introduced the von Neumann algebra $\mathcal N$ needed for a semifinite spectral triple.
Now we need the remaining ingredients.

We let $D=\frac{-1}{2\pi i}\frac{d}{dt}\otimes 1$ on $L^2({\R})\otimes H_\tau$ 
 so that $D$ is an unbounded self-adjoint operator affiliated to $\mathcal N$.

\begin{prop}
\label{prop:diff}
For $a\in{\rm dom}(\partial)$ we have $[D,\pi(a)]=\frac{1}{2\pi i}\pi(\partial(a)).$
\end{prop}

\begin{proof}
Let $\xi\in{\rm dom}(D)$ and $a\in{\rm dom}(\partial)$. Then we claim that
$\pi(a)\xi\in {\rm dom}(D)$. This follows from the computation
\begin{align*}
(D\pi(a)\xi)(s)&=(D\pi(a)\xi)(s)=\frac{-1}{2\pi i}\frac{d}{ds}\big(\pi(a)\xi\big)(s)
=\frac{-1}{2\pi i}\frac{d}{ds}\big(\alpha_s^{-1}(a)\xi\big)(s)\\
                  &=\frac{-1}{2\pi i}\Big(-\alpha_s^{-1}\big(\partial(a)\big)\xi(s)+\alpha_s^{-1}(a)\xi^\prime(s)\Big).
\end{align*}
Since $\big(\pi(a)D\xi\big)(s)=\alpha_s^{-1}(a)\big(\frac{-1}{2\pi i}\xi^\prime(s)\big)$, the result follows.
\end{proof}

To analyse functions of $D$, we first suppose that $A=\C$.
If we define the Fourier transform of a function $g\in L^1({\R})$ via
$
\hat{g}(s)=\int_{\R} e^{-2\pi its}g(t)dt,
$
then (provided $\hat{g}\in{\rm dom}(D)$) by a familiar calculation, 
$D(\hat{g}(t))=\widehat{tg(t)}$. Applying the functional calculus then yields
$$
f(D)\hat{g}=\widehat{fg}=\hat{f}*\hat{g}=\lambda(\hat{f})\hat{g}
$$ 
for  functions $f\in L^1({\R})$. In other words, 
$f(D)=\lambda(\hat{f}).$ In particular, if $f_s(t)=(1+t^2)^{-s/2}$ for $s>1$ so that $f_s\in L^1({\R})$
then $f_s(D)=(1+D^2)^{-s/2}=\lambda(\hat{f_s}).$ For general $A$, the same computations go 
through unchanged.

\begin{lemma}
\label{lem:H-S}
Let $h\in L^2({\R})$ be such that $\lambda(h)$ is in $\mathcal N$, and let 
$T\in M_{\bar\tau}^{1/2}$. Then  $T\lambda(h)$ is a 
Hilbert-Schmidt operator in $\mathcal{N}$ with respect to $\hat\tau$ and
moreover we have
$$
\hat\tau\big((T\lambda(h))^* T\lambda(h)\big)=\bar\tau(T^*T)\,\int_\R \overline h(t)\,h(t)\,dt.
$$
\end{lemma}


\begin{proof}
Firstly, by construction $T\lambda(h)\in\cn$.
Moreover, $x(t):=h(t) T$ and we have for  
$\xi\in C_c({\R},H_\tau)\subset L^2({\R},H_\tau)$:
\begin{align*}
(\tilde\pi(x)\xi)(s)&=\int_{\R}\alpha_s^{-1}(x(t))\xi(s-t)dt=\int_{\R}\alpha_s^{-1}(T)h(t)\xi(s-t)dt\\
&=\alpha_s^{-1}(T)\int_{\R} h(t)\xi(s-t)dt=\alpha_s^{-1}(T)(\lambda(h)\xi)(s)=(T\lambda(h)\xi)(s).
\end{align*}
By density of $C_c({\R},H_\tau)$ in $L^2({\R},H_\tau)$, we deduce that
$T\lambda(h)$  is of the form $\tilde\pi(x)$
for $x\in L^2(\R,H_\tau)$, since as previously remarked $L^2(\R,M_{\bar\tau}^{1/2})
\subset L^2(\R,H_\tau)$. The result follows by  equation \eqref{eq:trace-formula} since
$$
\bra x | x\ket=\int_{\R}\bar\tau(x(t)^*x(t))dt=\int_{\R}\bar\tau(T^*T\overline{h(t)}h(t))dt
=\bar\tau(T^*T)\int_{\R} |h(t)|^2 dt<\infty.
$$
\end{proof}

\begin{cor}
\label{roc}
Let $s>1$.
The restriction of the weight $\vf_s$ associated to $D$ (see Definition \ref{def:d-does-int})
to $M:=\pi(A)''\subset\cn:=\tilde\pi(A\rtimes_\alpha\R)''$ 
is proportional to $\overline\tau$, the normal extension of $\tau$ to $M$.
\end{cor}
\begin{proof}
By definition of $\vf_s$,  for $0\leq a\in M_{\bar\tau}$ and with $h_s(t):=(1+t^2)^{-s/4}$,
we have from Lemma \ref{lem:H-S} that
$$
\vf_s\big(\pi(a)\big)=\hat\tau\big(\lambda(\hat h_s)\pi(a)\lambda(\hat h_s)\big)
=\|h_s\|_2^2\,\tau(a)=\|h_s\|_2^2\,\overline\tau\big(\pi(a)\big).
$$ 
Hence,  $\vf_s|_{(M_\tau)_+}=\|h_s\|_2^2\,\overline\tau|_{(M_{\bar{\tau}})_+}$. 

Let now $T\in M\setminus M_{\bar{\tau}}$ satisfy $0\leq T$ so that $\bar{\tau}(T)=+\infty$. 
We construct a sequence $(T_k)_{k\in\N}$ in $M_{\bar{\tau}+}$
such that $T_k$ converges to $T$ in the weak operator topology and such that 
$0\leq T_k\leq T$. To do this, we choose $0\leq b_k\leq 1$ in $M_{\bar{\tau}}$
converging  in the weak operator topology to the identity of $M$ and set
$T_k:=T^{1/2}b_k T^{1/2}$. By construction, $0\leq T_k\leq T$, and since $M_{\bar{\tau}}$ is an ideal
in $M$, $T_k\in M_{\bar{\tau}}$. The convergence follows from  $\langle \psi,T_k\phi\rangle
=\langle T^{1/2}\psi,b_k T^{1/2}\phi\rangle\to \langle T^{1/2}\psi, T^{1/2}\phi\rangle
=\langle \psi, T\phi\rangle$, for all $\psi,\phi\in\H_\tau$.

Hence, we find that
$\vf_s(T)\geq\vf_s(T_k)=\|h_s\|_2^2\,\overline\tau\big(T_k\big)$,
and thus
$\vf_s(T)\geq\underline\lim_k \|h_s\|_2^2\,\overline\tau\big(T_k\big)$.
Since
the weak operator topology and the ultra-weak topology agree on bounded sets and 
$\overline\tau$ is ultraweakly lower semicontinuous, we deduce that
$\underline\lim_k \,\overline\tau\big(T_k\big)\geq \overline\tau\big(T\big)=+\infty$.
Hence $\vf_s(T)=+\infty$ and therefore $\vf_s|_{M_+}=\|h_s\|_2^2\,\overline\tau|_{M_+}$
as needed.
\end{proof}

{\bf Notation}. We use $({\rm dom}(\partial_\tau))^2$ for the $*$-algebra of finite sums of products of two elements in ${\rm dom}(\partial_\tau).$

\begin{lemma}
Let $a\in ({\rm dom}(\partial_\tau))^2 \subset {\rm dom}(\partial_\tau)$ so that $a$ is a sum of factors $b_ic_i$, where 
$b_i,\,c_i \in {\rm dom}(\partial_\tau)$.
Then for all $s>1$, the operator $\pi(a)(1+D^2)^{-s/2}$ is trace-class in 
$\mathcal{N}$ with respect to $\hat\tau$.
\end{lemma}

\begin{proof}
Without loss of generality we assume that $a=bc$ with $b,c \in {\rm dom}(\partial_\tau)$. Observe that
\begin{equation}
\pi(a)(1+D^2)^{-s/2}=\pi(bc)(1+D^2)^{-s/2}=\pi(b)[\pi(c),(1+D^2)^{-s/2}] + \pi(b)(1+D^2)^{-s/2}\pi(c).
\label{eq:first-line}
\end{equation}
The last term is trace-class in $\mathcal{N}$, 
since it is the product of two Hilbert-Schmidt operators in $(\mathcal{N},\hat\tau)$. 
Indeed, if we define
the bounded $L^2$-function, $f(t)=(1+t^2)^{-s/4}$, then
\begin{align*}
\pi(b)(1+D^2)^{-s/2}\pi(c)
&=\pi(b)(1+D^2)^{-s/4}(1+D^2)^{-s/4}\pi(c)=\pi(b)\lambda(\hat f)\lambda(\hat f)\pi(c)\\
&=\pi(b)\lambda(\hat f)\{\pi(c^*)\lambda(\hat f)\}^*=\tilde\pi(x)\tilde\pi(y)^*,
\end{align*}
where $x(t)=bf(t)$ and $y(t)=c^*f(t).$ So by the previous lemma 
$\tilde\pi(x)$ and $\tilde\pi(y)$ are Hilbert-Schmidt, and 
hence $\tilde\pi(x)\tilde\pi(y)^*$ is trace-class in $(\mathcal{N},\hat\tau)$.

We next show that the first term is trace-class in $(\mathcal{N},\hat\tau)$. 
This is more subtle. It suffices to assume that $s<2$, so that $s/2<1$. Let 
$C_s=\frac{\sin(s\pi/2)}{\pi}$ so by the integral formula for fractional powers,
\cite[page 701]{CP1}, we have
$
(1+D^2)^{-s/2}=C_s\int_0^\infty t^{-s/2}(1+D^2 +t)^{-1}dt.
$
Now we calculate, using the fact that $c\in{\rm dom}(\partial_\tau)$,
\begin{align*}
[\pi(c),(1+D^2)^{-s/2}]&=C_s\int_0^\infty t^{-s/2}[\pi(c),(1+D^2+t)^{-1}]dt\\
&=C_s\int_0^\infty t^{-s/2}(1+D^2+t)^{-1}[D^2,\pi(c)](1+D^2+t)^{-1} dt\\
&=C_s\int_0^\infty t^{-s/2}(1+D^2+t)^{-1}\Big(D[D,\pi(c)]+[D,\pi(c)]D\Big)(1+D^2+t)^{-1} dt\\
&=\frac{C_s}{2\pi i}\int_0^\infty t^{-s/2}(1+D^2+t)^{-1}\Big(D\pi(\partial_\tau(c))+\pi(\partial_\tau(c))D\Big)(1+D^2+t)^{-1} dt\\
&=\frac{C_s}{2\pi i}\int_0^\infty t^{-s/2}D(1+D^2+t)^{-1}\pi(\partial_\tau(c))(1+D^2+t)^{-1} dt\\
&\hspace*{1in}+\; \frac{C_s}{2\pi i}\int_0^\infty t^{-s/2}(1+D^2+t)^{-1}\pi(\partial_\tau(c))D(1+D^2+t)^{-1} dt.
\end{align*}

Hence, the first term on the right hand side of Equation \eqref{eq:first-line} equals
\begin{align*}
\pi(b)[\pi(c),(1+D^2)^{-s/2}]&=\frac{C_s}{2\pi i}\int_0^\infty t^{-s/2}\pi(b)D(1+D^2+t)^{-1}\pi(\partial_\tau(c))(1+D^2+t)^{-1} dt\\
&+\frac{C_s}{2\pi i}\int_0^\infty t^{-s/2}\pi(b)(1+D^2+t)^{-1}\pi(\partial_\tau(c))D(1+D^2+t)^{-1} dt.
\end{align*}
To complete the proof, we show that both of these integrands are trace-class in 
$\mathcal{N}$ and that the integrals converge in trace-norm. We do this for the 
first integral as the argument for the second integral is the same. We factor the 
integrand as a product of Hilbert-Schmidt operators and estimate their Hilbert-Schmidt
norms. 

First, $\pi(b)D(1+D^2+t)^{-1}=\pi(b)\lambda(\hat f)=\tilde\pi(b\cdot f)$ where 
$f(x)=\frac{x}{1+x^2+t}$ is a bounded $L^2$ function. Hence, and writing $\Vert\cdot\Vert_{HS}$, 
$\Vert\cdot\Vert_{op}$ for the Hilbert-Schmidt and operator norms respectively, 
\begin{align*}
\|\pi(b)D(1+D^2+t)^{-1}\|_{HS}&=\bra\tilde\pi(b\cdot f)\,|\,\tilde\pi(b\cdot f)\ket^{1/2}=\left(\tau(b^*b)\int_{\R}\left(\frac{x}{1+x^2+t}\right)^2dx\right)^{1/2}\\
&\leq \left(\tau(b^*b)\int_{\R}\frac{1}{1+x^2}\, dx\right)^{1/2}=\sqrt{\pi}\,\left(\tau(b^*b)\right)^{1/2}.
\end{align*}
Now letting $c^\prime = \partial(c)=\partial_\tau(c)\in A_\tau$,  we  calculate that
\begin{align*}
\|\pi(c^\prime)(1+D^2+t)^{-1}\|_{HS} &\leq \|\pi(c^\prime)(1+D^2+t)^{-1/2}\|_{HS}\|(1+D^2+t)^{-1/2}\|_{op}
\\
&=\left(\tau((c^\prime)^*c^\prime)\int_{\R}\left(\frac{1}{\sqrt{1+x^2+t}}\right)^2 dx\right)^{1/2}\|(1+D^2+t)^{-1/2}\|_{op}
\end{align*}\begin{align*}
&\leq  (\tau((c^\prime)^*c^\prime))^{1/2}\left(\int_{\R}\frac{1}{1+x^2} dx\right)^{1/2}\frac{1}{\sqrt{1+t}}
= \left(\tau((c^\prime)^*c^\prime))\,\pi\right)^{1/2}\frac{1}{\sqrt{1+t}}.
\end{align*}
Hence, the integrand is trace-class in $\mathcal{N}$ with trace-norm bounded by
$$
\| t^{-s/2}\pi(b)D(1+D^2+t)^{-1}\pi(\partial(c))(1+D^2+t)^{-1}\|_1\leq t^{-s/2}(\tau(b^*b)\,\tau((c^\prime)^*c^\prime))^{1/2}\pi\frac{1}{\sqrt{1+t}}.
$$
Since for $1<s<2$ the function $t\mapsto t^{-s/2}/\sqrt{1+t}$ is integrable as a function of $t\in [0,\infty)$, we see that the integral is a trace-class operator in $\mathcal{N}$.
\end{proof}

This completes the proof that 
 $({\rm dom}(\partial_\tau)^2, L^2(\R,\H_\tau),D)$ is a spectral triple.

We now extend our analysis with a useful formula for the trace of certain elements. First we need a technical result.
\begin{lemma}
If $\{A_n\}$ is a sequence of operators in $\mathcal{N}$ with 
$0\leq A_n \leq 1$ for all $n$ and $A_n\to 1$ in the weak 
operator topology on $\B(L^2({\R},H_\tau))$, then for all trace 
class operators $T\in\mathcal{N}$, $\hat\tau(T)=\lim_n \hat\tau(A_n T).$
\end{lemma}

\begin{proof}
Using the Jordan decomposition, it suffices to prove this for trace-class operators $T\geq 0$. In this case, 
$
T^{1/2}A_n T^{1/2}\leq T\|A_n\|\leq T,
$
and therefore $\hat\tau(A_n T)=\hat\tau(T^{1/2}A_n T^{1/2})\leq \hat\tau(T)$, so that $\overline{\lim}_n\hat\tau(A_n T)\leq\hat\tau(T)$.
On the other hand, one easily shows that $T^{1/2}A_n T^{1/2}\to T$ in the weak operator topology: that is, for $\xi,\eta\in L^2({\R},H_\tau)$
$$
\bra T^{1/2} A_n T^{1/2}\xi\,|\,\eta\ket=\bra A_n T^{1/2}\xi\,|\,T^{1/2}\eta\ket\to \bra T^{1/2}\xi\,|\,T^{1/2}\eta\ket=\bra T\xi\,|\,\eta\ket.
$$
Since the weak operator topology and the ultra-weak topology agree on bounded sets and $\hat\tau$ is 
ultraweakly lower semicontinuous, $\hat\tau(T)\leq \underline{\lim}_n\hat\tau(T^{1/2} A_n T^{1/2})=\underline{\lim}_n\hat\tau(A_n T)$. 
Thus we have the bounds
$$
\overline{\lim}_n\hat\tau(A_n T)\leq\hat\tau(T)\leq \underline{\lim}_n\hat\tau(A_n T),
$$ 
and the result follows.
\end{proof}

\begin{lemma}
\label{lem:first-integrable} Let $a\in ({\rm dom}(\partial_\tau))^2 \subset {\rm dom}(\partial_\tau)$.  Then for all $s>1$ 
$$
\hat\tau\left(\pi(a)(1+D^2)^{-s/2}\right)=\tau(a)\int_{\R} (1+t^2)^{-s/2}dt.
$$ 
\end{lemma}

\begin{proof}
Without loss of generality we assume that $a$ factors as $a=bc$, where $b,\,c \in {\rm dom}(\partial_\tau)$.
Let $g_n=\mathcal{X}_{[-n,n]}$ so that $g_n\in L^1({\R})\cap L^\infty({\R})\subset L^2({\R}).$ 
Then the multiplication operators $M_{g_n}$ on $L^2({\R})$ satisfy
$$
0\leq M_{g_n}\leq 1;\;\;\;\|M_{g_n}\|=1;\;\;\;M_{g_n}\to 1\;\;\;{\rm weakly\;\;on} \;\;L^2({\R}).
$$
Therefore, by the Fourier transform, we see that
$\hat{g}_n\in L^2({\R})\cap C_0({\R})$ and the convolution operators $\lambda(\hat{g}_n)$ satisfy
$$
0\leq \lambda(\hat{g}_n)\leq 1;\;\;\;\|\lambda(\hat{g}_n)\|=1;\;\;\;\lambda(\hat{g}_n)\to 1\;\;\;{\rm weakly\;\;on} \;\;L^2({\R}).
$$
By the previous two lemmas 
$\hat\tau(\pi(a)(1+D^2)^{-s/2})=\lim_n\hat\tau(\lambda(\hat{g}_n)\pi(bc)(1+D^2)^{-s/2})$. We let
$f_s(t)=(1+t^2)^{-s/2},$ so that 
$\hat\tau(\lambda(\hat{g}_n)\pi(bc)(1+D^2)^{-s/2})=\hat\tau(\lambda(\hat{g}_n)\pi(b)\pi(c)\lambda(\hat{f}_s))$.
If we define $x_n(t)=b^*\hat{g}_n(t)$ and 
$y(t)=c\hat{f}_s(t)$ then $\tilde\pi(x_n)=\pi(b^*)\lambda(\hat{g}_n)$ and $\tilde\pi(y)=\pi(c)\lambda(\hat{f}_s)$
so that
\begin{align*}
&\hat\tau\left(\lambda(\hat{g}_n)\pi(bc)(1+D^2)^{-s/2}\right)
=\hat\tau\left(\tilde\pi(x_n)^*\tilde\pi(y)\right)=\int_{\R}\tau\left(x_n(t)^* y(t)\right)dt
=\int_{\R}\tau\left(\overline{\hat{g}_n(t)}bc\hat{f}_s(t)\right)dt\\
&=\tau(bc)\int_{\R}\overline{\hat{g}_n(t)}\hat{f}_s(t)dt
=\tau(a)\int_{\R}\overline{g_n(t)}f_s(t)dt=\tau(a)\int_{-n}^n(1+t^2)^{-s/2}dt.
\end{align*}
\begin{align*}\mbox{Hence, }\quad\quad\quad\quad
\hat\tau(\pi(a)(1+D^2)^{-s/2})&=\lim_n\tau(a)\int_{-n}^n(1+t^2)^{-s/2}dt=\tau(a)\int_{\R} (1+t^2)^{-s/2}dt.\quad\quad\quad\qedhere
\end{align*}
\end{proof}

\begin{cor}
\label{cor:izza-cor}
Let $a\in ({\rm dom}(\partial_\tau^2))^2 \subset {\rm dom}(\partial_\tau)$ so that 
$a$ is a sum of factors $bc$, where $b,\,c \in {\rm dom}(\partial_\tau^2)$. Then with 
$e=1+a$ invertible in $A^\sim$,
\begin{align*}{\rm Res}_{s=1}\left\{\frac{1}{2}\hat\tau(\pi(e^{-1})[D,\pi(e)](1+D^2)^{-s/2})\right\}
&=\lim_{s\to 1} \frac{1}{2}(s-1)\hat\tau(\pi(e^{-1})[D,\pi(e)](1+D^2)^{-s/2})\\
&=\frac{1}{2\pi i}\tau(e^{-1}\partial(e)).
\end{align*}
\end{cor}
\begin{proof}
It suffices to see that $e^{-1}\partial_\tau(e)$ is a finite sum of products 
satisfying the hypotheses of the previous lemma. To this end let $e^{-1}=1-f$
where $f\in A_\tau$. Then
\begin{align*}
e^{-1}\partial_\tau(e)&=(1-f)\partial_\tau(1-bc)
=(f-1)\{\partial_\tau(b) c + b\partial_\tau(c)\}\\
&=- \partial_\tau(b) c - b\partial_\tau(c)+f\partial_\tau(b) c +f b\partial_\tau(c),
\end{align*}
and we note that each left factor
$\partial_\tau(b),b,f\partial_\tau(b),fb \in A_\tau\subset A_\tau^{1/2}$; and each right factor 
$c,\partial_\tau(c)\in{\rm dom}(\partial_\tau)$.
It follows from Proposition \ref{prop:diff} and Lemma \ref{lem:first-integrable} that 
\begin{align*}
\hat\tau\left(\pi(e^{-1})[D,\pi(e)](1+D^2)^{-s/2}\right)
&=\frac{1}{2\pi i}\tau(e^{-1}\partial(e))\int_{\R}(1+t^2)^{-s/2} dt.
\end{align*}
The result follows 
from the fact that ${\rm Res}_{s=1}\int_{\R}(1+t^2)^{-s/2} dt=2$.
\end{proof}
\subsection{Connection with noncommutative integration theory  and the smoothness question}

The remainder of this Section is devoted to explaining how this example fits with the formulation of the nonunital  local index formula as proved in \cite{CGRS2}.
In other words we will prove a version of the Phillips-Raeburn index theorem.
Recall now the notation from Section 2.

\begin{prop}
\label{prop:integrable}
With $(A,\tau,\alpha)$ as above and defining $\B_1(D,1)$ and $\B_2(D,1)$ relative to $(\cn,\,\hat\tau)$, 
we have 
$\pi(A)\cap \B_2(D,1)=\pi(A_\tau^{1/2})$ and $\pi(A)\cap \B_1(D,1)=\pi(A_\tau)$.
\end{prop}

\begin{proof}
First, Lemma \ref{lem:H-S} shows that $\pi(A_\tau^{1/2})\subset \B_2(D,1)$. Conversely, if $a\in A$ and
also $\pi(a)\in \B_2(D,1)$, then by definition $\pi(a)(1+D^2)^{-s/4}\in \L^2(\cn,\hat\tau)$ for all $s>1$. As before we 
write $M=\pi(A)''\subset \B(\H_\tau)$ and $\bar\tau$ for the normal extension of $\tau$ to $M$ and since
$\L^2(\cn,\hat\tau)=L^2(\R,\L^2(M,\bar\tau)),$  
we have $\pi(a)\in M_{\bar\tau}^{1/2}.$  Hence $\pi(a) \in M_{\bar\tau}^{1/2}\cap \pi(A) =\pi(A^{1/2}_\tau)$ . 
Thus $\pi(A)\cap \B_2(D,1)=\pi(A_\tau^{1/2})$.

For the final statement, we recall the result in  Corollary \ref{roc} together with the notation given there. Combining
this with \cite[Proposition 1.19]{CGRS2}, we deduce that
$$
\B_1(D,1)\cap \pi(A)''=\bigcap_{s> 1} {\rm dom}(\|h_s\|_2^2\overline\tau)=
{\rm dom}(\overline\tau)=\pi(M_{\bar{\tau}}).
$$ 
Taking the intersection with $\pi(A)$ gives $\B_1(D,1)\cap \pi(A)=\pi(A_\tau)$ as needed.
%
%
%
\end{proof}

The argument of the previous proposition analyses the  integration theory that forms the first ingredient for the local index formula.
 What remains is to find a subalgebra
of ${\rm dom}(\partial_\tau)\subset A$ which yields a smoothly summable spectral triple
in the sense of Definition \ref{delta-phi}.

We recall from Definition \ref{def:define-all} the (partially defined) operators 
$\cn\ni T\mapsto L(T):=(1+D^2)^{-1/2}[D^2,T]$ and $\cn\ni T\mapsto R(T):=[D^2,T](1+D^2)^{-1/2}$.
Also we set $F_D=D(1+D^2)^{-1/2}$.

\begin{lemma}
If $a\in {\rm dom}(\partial^2)$ then $\pi(a)\in {\rm dom}(L)\cap {\rm dom}(R)$ and on the space 
$H_\infty=\cap{\rm dom}(D^k)$
\begin{eqnarray*}
L(\pi(a))&=&\frac{1}{\pi i}F_D\pi(\partial(a)) +\frac{1}{4\pi^2}(1+D^2)^{-1/2}\pi(\partial^2(a))\;\;{\rm and}\\
R(\pi(a))&=&\frac{1}{\pi i}\pi(\partial(a))F_D + \frac{1}{4\pi^2}\pi(\partial^2(a))(1+D^2)^{-1/2}.
\end{eqnarray*} 
\end{lemma}

\begin{proof}
The following calculation takes place on $H_\infty=\cap_k{\rm dom}(D^k)$ where we may 
commute $D$ with bounded functions of $D$. The calculation for $R$ is similar
as $R(\pi(a))^*=-L(\pi(a^*))$.
\begin{align*}
L(\pi(a))&=(1+D^2)^{-1/2}[D^2,\pi(a)]=(1+D^2)^{-1/2}\{D[D,\pi(a)]+[D,\pi(a)]D\}\\
&=\frac{1}{2\pi i}F_D\pi(\partial(a))+\frac{1}{2\pi i}(1+D^2)^{-1/2}\pi(\partial(a))D\\
&=\frac{1}{2\pi i}F_D\pi(\partial(a))+\frac{1}{2\pi i}(1+D^2)^{-1/2}\left([\pi(\partial(a)),D]+D\pi(\partial(a))\right)\\
&=\frac{1}{\pi i}F_D\pi(\partial(a))+\frac{1}{4\pi^2}(1+D^2)^{-1/2}\pi(\partial^2(a)).\qedhere
\end{align*}
\end{proof}

\begin{prop}
\label{prop:smooth}
If $a\in\bigcap_{n=1}^{\infty} {\rm dom}(\partial^n)$ then 
$\pi(a)\in \bigcap_{l,k}{\rm dom}({R}^l\circ {L}^k)$. 
Hence, by the equality of $\bigcap_{l,k}{\rm dom}({R}^l\circ {L}^k)$ and 
$\bigcap_{n=1}^{\infty} {\rm dom}(\delta^n)$, see \cite{CPRS2}, 
if $a\in A$ is smooth in the sense of the action $\alpha$
of $\R$ on $A$ then $\pi(a)$ is smooth in the sense of the derivation $\delta.$
\end{prop}

\begin{proof}
It suffices to prove the following fact by induction on 
$n=l+k$: if $a\in\bigcap_{j=1}^{\infty} {\rm dom}(\partial^j)$ 
then ${R}^l\circ {L}^k(\pi(a))$ is a finite sum of 
terms of the form $g(D)\pi(b)f(D)$ where $g,f$ 
are continuous bounded functions on $\R$ and 
$b=\partial^m(a)$ is a smooth element in $A$ with $m\leq 2n$.

When $n=1$ we are looking at 
$L(\pi(a))$ or $R(\pi(a))$ which have the 
correct form by the previous lemma. Now if the 
result holds for some $n=(l+k)\geq 1$ then 
we obtain the case $n+1$ by applying either 
$L$ or $R$ to this case since $L$ 
and $R$ commute. By the inductive hypothesis
it suffices to apply $L$ or $R$ to a term of the 
form $g(D)\pi(b)f(D).$ We apply 
$L$ as the other case is similar. A computation like those above yields
\begin{eqnarray*}
L(g(D)\pi(b)f(D))&=&g(D)\left\{\frac{1}{\pi i}F_D\pi(\partial(b))+\frac{1}{4\pi^2}(1+D^2)^{-1/2}\pi(\partial^2 (b))\right\}f(D).
\end{eqnarray*}
Since $b=\partial^m(a)$,   $\partial(b)=\partial^{m+1}(a)$  and  
$\partial^2(b)=\partial^{m+2}(a)$, the induction is complete
\end{proof}

\begin{rem*}
{\rm Proposition \ref{prop:smooth} shows that with $\A\subset A$ the smooth elements for the action of $\alpha$,
the  spectral triple $(\A,L^2(\R,\H_\tau),D)$ is $QC^\infty$ or {\bf smooth}. 
However we need more than this to deal 
with integrability as well as smoothness. The next result combines our smoothness and integrability results,
and recovers the Phillips-Raeburn and Lesch index theorems.}
\end{rem*}

\begin{thm}
Let ${\mathcal C}\subset A_\tau$ be the $*$-algebra generated by
$$\{ab\in A_\tau: \partial^k_\tau(a),\,\partial^k_\tau(b)\in A_\tau^{1/2}\ \ {\rm for\ all}\ k=0,1,2,\dots\}.
$$
Then $({\mathcal C},L^2(\R,H_\tau),D)$ is a smoothly summable semifinite 
spectral triple relative to $({\mathcal N},\hat\tau)$ with
spectral dimension 1. The spectral dimension is isolated and the formula
$$
{\mathcal C}\ni a_0,a_1\mapsto \phi_1(a_0,a_1)
:=\frac{1}{2}{\rm Res}_{s=1}\hat\tau(a_0[D,a_1](1+D^2)^{-s/2})
$$
defines a $(b,B)$ cocycle for ${\mathcal C}$. 
Moreover, for $P=\chi_{[0,\infty)}(D)$ and
$u=1+a$ unitary with $a\in {\mathcal C}$,
$$
{\rm Index}_{\hat\tau}(PuP)=- \phi_1(u^*,u)=- \frac{1}{2\pi i}\,\tau(u^*\partial(u)).
$$
\end{thm}


\begin{proof}
From Proposition \ref{prop:integrable}, each $\pi(\partial_\tau^k(a))$, $a\in \mathcal C$ is an element of 
$\B_1(D,1)$, and hence $\pi({\mathcal C})\subset \B_1^\infty(D,1)$. Also since $[D,\pi({\mathcal C})]\subset \pi({\mathcal C})$, we have $[D,\pi({\mathcal C})]\subset \B_1^\infty(D,1)$.
By \cite[Proposition 3.16]{CGRS2}, the spectral triple is smoothly summable 
with spectral dimension 1. That the spectral dimension is isolated follows
from the fact that only one zeta function arises in the local index formula, 
and so (see \cite{CGRS2}) is guaranteed to have
at worst a simple pole at $s=1$. All the remaining claims follow from Corollary \ref{cor:izza-cor} and
the proof of the local index formula in \cite{CGRS2}.
\end{proof}

Our result here shows that an important class of examples fall into the framework of \cite{CGRS2}.
Notice that in this case our formula involves the path $D+tu[D,u^*]$ which 
is generically not a path of (Breuer-)Fredholm operators. The same issue arises in general as can be seen from \cite{CGRS2} and the resolution of this 
apparent difficulty
in general will be to replace this path by one in the `double' which is introduced in the next Section.  Using the double it is straightforward to prove as in \cite{CGRS2} that
$
\chi_{[0,\infty)}(D)-\chi_{[0,\infty)}(uDu^*)\in \K(\cn,\hat\tau).
$
Hence $\chi_{[0,\infty)}(D)$, $\chi_{[0,\infty)}(uDu^*)$ form a Fredholm pair. What we have done in this Section is show that
$$
{\rm Index}(\chi_{[0,\infty)}(D)\chi_{[0,\infty)}(uDu^*))={\rm Index}(\chi_{[0,\infty)}(D)\,u\,\chi_{[0,\infty)}(D))
$$
is given by the local index formula as a residue that is recognisably the Phillips-Raeburn-Lesch formula, \cite{L,PR}.

In the next Section we will attempt to generalise this strategy, namely to go from the local index formula
to a spectral flow formula for paths of the form $D+tu[D,u^*]$.
As noted in the introduction, we leave open the possibility of a definition and computation of spectral
flow for paths in our affine space of perturbations of $\D$ where the endpoints
are not unitarily equivalent.
%

\section{From the resolvent cocycle to the spectral flow formula}
\label{sec:spec-flow}

\subsection{Spectral flow and the index}

To place our results in their proper setting we need some
background from \cite{Ph1,Ph2,PR,CPRS3}. Let
$\pi:{\mathcal N}\to {\mathcal N}/{\mathcal K_{\cn}}$
be the canonical mapping onto the Calkin algebra. A Breuer-Fredholm operator is one that
maps to an invertible operator under $\pi$.  The theory of Breuer-Fredholm operators  
for the case where $\cn$ is not a factor is developed
in \cite{PR,CPRS3}, by analogy with the 
factor case of Breuer, \cite{B1,B2}. 

We say that an unbounded densely defined self-adjoint operator $D$ on $\HH$ is 
a Breuer-Fredholm operator  if $F_D:=D(1+D^2)^{-1/2}$ is 
 Breuer-Fredholm in $\mathcal N$.
Recall that the Breuer-Fredholm index of
a Breuer-Fredholm operator $F$ is defined by
$$
{\rm Index}_\tau(F)=\tau(Q_{\ker F})-\tau(Q_{{\rm coker} F})
$$
where $Q_{\ker F}$ and $Q_{{\rm coker} F}$ are the projections onto
the kernel and cokernel of $F$. The Breuer-Fredholm index is in general real-valued.

We use the function $sign$ defined by $sign(t)=1$ for $t\geq 0$ and $sign(t)=-1$ for $t<0$.

\begin{defn}
If $\{F_t\}$ is a norm continuous path of self-adjoint Breuer-Fredholm
operators
in $\mathcal N$, then the definition of the {\bf spectral flow} of the
path, $sf_\tau(\{F_t\})$ is based on the following sequence of observations
in \cite{Ph2}:

\noindent 1. The function $t\mapsto sign(F_t)$ is typically discontinuous as is
the
projection-valued mapping $t\mapsto P_t=\frac{1}{2}(sign(F_t)+1)$.

\noindent 2. However,  $t\mapsto \pi(P_t)$ is norm continuous.

\noindent 3. If $P$ and $Q$ are projections in $\mathcal N$ and
$||\pi(P)-\pi(Q)||<1$
then $PQ:Q\HH \to P\HH$ is a Breuer-Fredholm operator and so
${\rm Index}_\tau(PQ)\in {\R}$ is well-defined. (This needs
\cite[Section 3]{CPRS3}.)

\noindent 4. If we partition the parameter interval of $\{F_t\}$ so
that the $\pi(P_t)$ do not vary much in norm on each subinterval
of the partition then 
$
sf_\tau(\{F_t\}):=\sum_{i=1}^n
{\rm Index}_\tau(P_{t_{i-1}}P_{t_i})
$
is a well-defined and (path-) homotopy-invariant real number which
agrees with the usual notion of spectral flow in the type $I_\infty$
case.

\noindent 5. Let $\{D_t\}$ be a path of unbounded Breuer-Fredholm operators such that 
the path $\{(F_D)_t\}$ is a norm continuous path of Breuer-Fredholm operators. 
We define the spectral flow of the path $\{D_t\}$ to be the spectral flow of the path
$\{(F_D)_t\}$.  We  observe that this is an integer in the
$I_\infty$ case and a real
number in the general semifinite case.
\end{defn}

Fix an unbounded self-adjoint Breuer-Fredholm operator $D$, and
let $P$ denote the projection onto the non-negative spectral subspace of $D$.
Suppose that $u$ is a unitary in $\cn$ such that $D_t:=D+tu[D,u^*]$ is a path of Breuer-Fredholm
operators such
that $F_t:=F_{D+tu[D,u^*]}$ is a norm continuous path and $F_1-F_0$ is compact. 
In this special case we denote the spectral flow by 
$
sf_\tau(D,uDu^*):=sf_\tau(\{D_t\}):=sf_\tau(\{F_t\}).
$
That is, the spectral flow along $\{D_t\}$ is defined to be $sf_\tau(\{F_t\})$
and by \cite{CP1} this is the Breuer-Fredholm index of
$PuPu^*$. (Note that $sign(F_1)=2uPu^*-1$ and 
since we assume $sign(F_1)-sign(F_0)=2(uPu^*-P)$ is compact,
$PuPu^*$ is certainly Breuer-Fredholm from 
$uPu^*\HH \to P\HH$.) 
Now, \cite[Appendix B]{PR}, we have ${\rm Index}_\tau(PuPu^*)={\rm Index}_\tau(PuP)$.
Hence $sf_\tau(D,uDu^*)={\rm Index}_\tau(PuP)$.

All of the above works well when we have $(1+D^2)^{-1/2}$ compact, for then
one can show that with $D_t=D+tu[D,u^*]$, the path $F_t$ is indeed a continuous 
path of Breuer-Fredholms. When the resolvent of $D$ is not compact, we require additional
assumptions, as in the next result.

\begin{theorem}
\label{SF=IND}
Let $(\A,\HH,\D)$  be an odd nonunital semifinite spectral triple relative to $(\cn,\tau)$
(no smoothness assumptions) with {\bf $\D$ invertible}  and let $\A^\sim$ denote $\A$ with a unit adjoined. 
Let $u\in\A^\sim$ be a unitary
such that 
$[\D,u](1+\D^2)^{-1/2}$ is compact.
Setting  $P=\chi_{[0,\infty)}(\D)$, we have
$$
sf_\tau(\D, u\D u^*)= {\rm Index }_{\tau} \big(P\,uP\big).
$$ 
\end{theorem}

\begin{proof}
As $\D$ is invertible, it is Fredholm.
Also
the bounded operator $F_\D= \D(1+\D^2)^{-1/2}$ is invertible and so Fredholm.
Consider the difference
$$
F_\D-F_{u\D u^*}=
\D(1+\D^2)^{-1/2}-(u\D u^*-\D)(1+u\D^2u^*)^{-1/2}
-\D(1+u\D^2u^*)^{-1/2}.
$$
The middle term in this expression is
$[\D,u](1+\D^2)^{-1/2}u^*$ which is compact by assumption.
Next we combine
the remaining two terms in the previous equation as
$$
\D[(1+\D^2)^{-1/2}- (1+u\D^2u^*)^{-1/2}],
$$
and use the integral formula for fractional powers from \cite{CP1} to obtain
$$
\D[(1+\D^2)^{-1/2}- (1+u\D^2u^*)^{-1/2}]=
\frac{1}{\pi}
\D\int_0^\infty \lambda^{-1/2}[(\lambda+1+\D^2)^{-1}-
(\lambda+1+u\D^2u^*)^{-1}]d\lambda
$$
$$
=\frac{1}{\pi} \int_0^\infty \lambda^{-1/2}
\D(\lambda+1+\D^2)^{-1}(\D u[\D,u^*]+u[\D,u^*]\D +(u[\D,u^*])^2)
u(\lambda+1+\D^2)^{-1}u^*\,d\lambda.
$$
By  \cite[Lemma 6, part (2), Appendix A]{CP1}, the 
integral above converges in norm. Since $u[\D,u^*]u=-[\D,u]$, 
the three operator terms in the integrand can be written as:
$$
\D^2(\lambda+1+\D^2)^{-1}u[\D,u^*]u(\lambda+1+\D^2)^{-1}u^*
=-\D^2(\lambda+1+\D^2)^{-1}[\D,u](\lambda+1+\D^2)^{-1}u^*;
$$
$$
\D(\lambda+1+\D^2)^{-1}u[\D,u^*][\D,u](\lambda+1+\D^2)^{-1}u^*
-\D(\lambda+1+\D^2)^{-1}[\D,u](\lambda+1+\D^2)^{-1}\D u^*;
$$
$$
-\D(\lambda+1+\D^2)^{-1}u[\D,u^*][\D,u](\lambda+1+\D^2)^{-1}u^*.
$$
Inspection now shows that each of these terms
contains a resolvent times  $[\D,u]$, which by our assumptions is compact.
Thus the integrand is compact.
So $F_\D-F_{u\D u^*}$ is compact since the integral converges in norm. 
Therefore the spectral flow is indeed given
as the index of $PuP$ by the definitions and results in \cite{Ph2} (see also \cite{BCPRSW} for the extension to the non-factor case).
\end{proof}

We can always put ourselves into the situation where $\D$ is invertible, using the
following doubling construction due originally to Connes \cite{Co1}.

\begin{definition} 
\label{def:double}
Let $(\A,\HH,\D)$ be a  semifinite spectral triple relative to $(\cn,\tau)$. For any
$\mu>0$,  define the `double' of
$(\A,\HH,\D)$ to be the  semifinite
spectral triple $(\A,\HH^2,\D_\mu)$ relative to $(M_2(\cn),\tau\otimes\tr_2)$, 
with $\HH^2:=\HH\oplus\HH$ and with
the action of $\A$ and $\D_\mu$ given by
\ben
\D_\mu:=\bma \D & \mu\\ \mu & -\D\ema,\ \ \ \ a\mapsto \hat a:=\bma a & 0\\ 0 & 0\ema,
\ \ \forall a\in\A.
\een
\end{definition}

{\bf Remark.} Whether $\D$ is invertible or not, $\D_\mu$ always is invertible,
and $F_\mu=\D_\mu|\D_\mu|^{-1}$
has square 1. This is the chief reason for introducing this construction.

We also need to extend the action of $M_n(\A^\sim)$ on $(\H\oplus\H)\otimes\C^n$,  
in a compatible way with the extended action of $\A$ on $\H\oplus\H$. So, for a generic 
element $b\in M_n(\A^\sim)$,  we let
\begin{equation}
\label{ext-A-sim}
\hat b:= \bma b & 0\\ 0 & {\bf 1}_b\ema\in M_{2n}(\cn),
\end{equation}
with ${\bf 1}_b:=\rho^n(b)\,{\rm Id}_{M_n(\cn)}$, where $\rho^n:M_n(\A^\sim)\to M_n(\C)$ is the quotient map.

The index pairings of $(\A,\H,\D)$ and $(\A,\H^2,\D_\mu)$ with $K_*(\A)$ agree. This is
proved in \cite[Section 3]{CGRS2}, and more information can be found there.

%
%

Let $u\in M_n(\A^\sim)$ be a unitary, and suppose that 
$[\D\otimes{\rm Id}_n,u](1+\D^2\otimes{\rm Id}_n)^{-1/2}$ is compact. 
This implies that 
$[\D_\mu\otimes{\rm Id}_n,\hat{u}](1+\D_\mu^2\otimes{\rm Id}_n)^{-1/2}$ is compact also.
Together with the fact that $\D_\mu$ is invertible, 
the spectral flow from $\D_\mu\otimes {\rm Id}_n$ to $\hat u(\D_\mu\otimes {\rm Id}_n) \hat u^*$
is well-defined, by Theorem \ref{SF=IND}. 
Consequently if $P_\mu$ is the spectral projection of $\D_\mu$ corresponding to 
the interval $[0,\infty)$ then 
$$
{\rm Index}_{\tau\otimes\tr_2\otimes\tr_n}(P_\mu \hat{u} P_\mu)= 
sf_{\tau\otimes\tr_2\otimes\tr_n}(\D_\mu, \hat{u}\D_\mu \hat{u}^*).
$$

\begin{cor}
\label{SF=nicer-index}
Let $(\A,\HH,\D)$  be an odd nonunital semifinite spectral triple relative to $(\cn,\tau)$
(no smoothness assumptions)  and let $\A^\sim$ denote $\A$ with a unit adjoined. 
Let $u\in M_n(\A^\sim)$ be a unitary
such that 
$[\D\otimes {\rm Id}_n,u](1+\D^2\otimes{\rm Id}_n)^{-1/2}$ is compact.
Setting  $P=\chi_{[0,\infty)}(\D)$, we have
\begin{align*}
{\rm Index}_{\tau\otimes\tr_2\otimes\tr_n}(P_\mu\otimes{\rm Id}_n \hat{u} P_\mu\otimes{\rm Id}_n)
&= sf_{\tau\otimes\tr_2\otimes\tr_n}(\D_\mu\otimes{\rm Id}_n, \hat{u}(\D_\mu\otimes{\rm Id}_n) \hat{u}^*)\\
&= {\rm Index }_{\tau\otimes \tr_n} \big(\,(P\otimes{\rm Id}_n)\,u\,(P\otimes{\rm Id}_n)\,\big).
\end{align*}
\end{cor}

\begin{proof}
This follows from the comments above and \cite[Proposition 3.25]{CGRS2} 
where the final equality is proved.
\end{proof}

For the next few definitions and results, we assume that we have a nonunital smoothly 
summable spectral triple $(\A,\HH,\D)$, relative to $(\cn,\tau)$, 
and of spectral dimension $p\geq 1$. This implies that 
$[\D,u](1+\D^2)^{-1/2}$ is compact, and similarly with $\D$ replaced by $\D_\mu$, so
that the spectral flow of $t\mapsto \D_\mu+tu[\D_\mu,u^*]$ is well-defined.

In this context we can define the resolvent cocyle, which can be used to compute the
pairing with $K$-theory. 

\begin{definition}\label{expectation} For $0<a<1/2$, let $\ell$ be the vertical line
$\ell=\{a+iv:v\in\R\}$.
Given $m\in\mathbb N$, $s\in\mathbb R^+$, $r\in\mathbb C$ 
and operators $A_0,\dots,A_m\in {\rm OP}^{k_i}(\D)$ and 
$A_0\in{\rm OP}^{k_0}_0(\D)$, such that   $|k|-2m<2\Re(r)$, we  define
\begin{align}
\label{expect1}
 \langle A_0,\dots,A_m\rangle_{m,r,s}&:=\frac{1}{2\pi i}\,\tau\Big(\gamma\int_\ell
\lambda^{-p/2-r}
A_0\,R_{s}(\lambda)\cdots A_m\,R_{s}(\lambda)\,d\lambda\Big),
\end{align}
Here $\gamma$ is the ${\mathbb Z}_2$-grading in the even case and the identity operator
in the odd case, and $R_s(\lambda)=(\lambda-(1+s^2+\D^2))^{-1}$.
\end{definition}

We now state the definition of the resolvent cocycle for odd spectral triples in terms of the
expectations  $\langle\cdot,\dots,\cdot\rangle_{m,r,s}$. Let 
$N:=\lfloor p/2\rfloor+1$ and $M:=2N-1$.

\begin{definition}
\label{resolvent}
For   $m= 1,\,3,\dots,M$, we introduce the constants  $\eta_m$ by
\begin{align*}
 \eta_m&=\left(-\sqrt{2i}\right)2^{m+1}\frac{\Gamma(m/2+1)}{\Gamma(m+1)}.
  \end{align*}
Then for  $\Re(r)>(1-m)/2$, the $m$-th component of the
{\bf resolvent cocycle} $\phi_m^r:\A\otimes{\A}^{\otimes m}\to{\C}$ is defined by
\begin{align*}
\phi_{m}^r(a_0,\dots,a_m)&:=\eta_m\int_0^\infty s^m\langle
a_0,da_1,\dots,da_m\rangle_{m,r,s}\,ds,
\end{align*}
\end{definition}

\begin{rem*}
{\rm It is important to note that the resolvent cocycle $\phi^r_m$ is well 
defined even when $\D$ is not invertible. This follows from \cite[Lemma 4.3]{CGRS2}.}
\end{rem*}

To state our main theorem we need the definition of the Chern character of a unitary.
The (infinite) $b,B$-cycle $Ch(u)=(Ch_{2j+1}(u))_{j\geq0}$ of $u\in M_n(\A)$ is given by
$$
Ch_{2j+1}(u)=(-1)^jj!\sum_{i_0,i_1,\dots,i_{2j+1}}(u^*)_{i_0,i_1}\otimes (u)_{i_1,i_2}
\otimes (u^*)_{i_2,i_3}\otimes\cdots\otimes (u)_{i_{2j+1},i_0}\ \ \ 
(2j+2\ \ \mbox{entries}).
$$
We refer to \cite{CPRS2} for more information in this context.

\begin{theorem}
\label{thm:sf=cocycle}
Let $(\A,\HH,\D)$  be an odd nonunital smoothly summable 
semifinite spectral triple relative to $(\cn,\tau)$, 
and let $\A^\sim$ denote $\A$ with a unit adjoined. Let $u\in M_n(\A^\sim)$ be a unitary. Then
\begin{align*}
{\rm Index }_{\tau\otimes\tr_n} \big(P\,uP\big)
&=sf_{\tau\otimes\tr_2\otimes\tr_n}(\D_\mu, \hat{u}\D_\mu \hat{u}^*)
=\frac{-1}{\sqrt{2\pi i}}{\rm Res}_{r=(1-p)/2}\sum_{m=1,odd}^M\phi_m^r({\rm Ch}_m(u))\\
&=\frac{-1}{\sqrt{2\pi i}}\frac{1}{2}{\rm Res}_{r=(1-p)/2}\sum_{m=1,odd}^M
\phi_m^r({\rm Ch}_m(u)-{\rm Ch}_m(u^*)).
\end{align*}
In particular, the residues exist. 
\end{theorem}

\begin{proof}
The first equality has already been discussed. The second equality is 
the nonunital local index formula, \cite[Theorem 4.33]{CGRS2}, and the third equality is 
again the local index formula together with the fact that the (entire) $(b,\,B)$ cycle 
$\big({\rm Ch}_m(\hat{u})+{\rm Ch}_m(\hat{u}^*)\big)_{m=1,3,\dots}$ is a boundary 
(see \cite[Lemma 3.1]{CPRS2} for a 
simple proof).
\end{proof}

\subsection{The statement of the main result}

Our main result shows that we can obtain a  formula analogous to that of \cite{CP1,CP2} for 
the paths we are considering.

\begin{theorem}\label{specflow}
Let $(\A,\HH,\D)$  be an odd nonunital 
smoothly summable spectral triple, relative to $(\cn,\tau)$, of 
spectral dimension $p\geq 1$. Let $F_\mu=\D_\mu|\D_\mu|^{-1}$ 
(where $\D_\mu$ comes from the double picture), 
$P_\mu=(1+F_\mu)/2$ and $P=\chi_{[0,\infty)}(\D)$.
Then for any unitary $u\in M_n(\A^\sim)$ we have the equalities
\begin{align*}
&{\rm Index}_{\tau\otimes\tr_2\otimes\tr_n}\big(({P}_\mu\otimes {\rm Id}_n) \hat u({P}_\mu\otimes {\rm Id}_n)\big)
={\rm Index}_{\tau\otimes\tr_n}\big(({P}\otimes {\rm Id}_n) u({P}\otimes {\rm Id}_n)\big)\\
&={\rm Res}_{z=0}\int_0^1\tau\otimes 
\tr_n\Big(u[\D\otimes {\rm Id}_n,u^*]\big(1+(\D\otimes {\rm Id}_n+tu[\D\otimes {\rm Id}_n,u^*])^2\big)^{-1/2-z}\Big)dt\,.
\end{align*}
\end{theorem}

To prove this Theorem, we are going to follow closely some aspects of the 
argument of \cite[section 5.3]{CPRS2}.

We fix the following data for the remainder of this Section.
Let $(\A,\HH,\D)$  be an odd nonunital semifinite spectral triple relative to $(\cn,\tau)$, smoothly summable
with spectral dimension $p\geq 1$.
To simplify the discussion, we restrict ourselves to the case where  
$u$ is a unitary in $ \mathcal A^\sim $ and not in $M_n(\A^\sim)$.
(The general $u\in M_n(\A^\sim)$ case, will follow
by replacing $\D$ by $\D\otimes {\rm Id}_n$
and  $\cn$ by $M_n(\cn)$ in all the formulas below.) 

\subsection{Notation and basic results for exploiting Clifford periodicity}
The idea behind the construction in this Section comes from \cite{G}.
We use however the analytic formulation in \cite{CP2, CPRS2}.

We form the Hilbert space
$\tilde{\mathcal H}:=\C^2\otimes\C^2\otimes\HH$ acted upon by the
von Neumann algebra,
$\tilde{\mathcal N}:=M_2(\C)\otimes M_2(\C)\otimes \mathcal N$. 
Note that $\tilde\cn$ is naturally endowed
with the normal semifinite faithful trace $\tilde\tau:=\tr_4\otimes\tau$. 
Introduce the two dimensional Clifford algebra, with generators in the form (Pauli matrices)
$$
\sigma_1:= \left(\begin{array}{cc}
0 & 1 \\
1 & 0
\end{array} \right),\ \ \sigma_2:= \left(\begin{array}{cc}
0 & -i \\
i & 0
\end{array} \right), \ \ \sigma_3:=
\left(\begin{array}{cc} 1 & 0 \\
0 & -1
\end{array} \right).
$$
Define the grading on $\tilde{\mathcal H}$ by
$\Gamma := \sigma_2\otimes\sigma_3\otimes {\rm Id}_\cn\in \tilde{\mathcal N}$.

For $t\in [0,1]$ and $ s\in [0,\infty)$, introduce the even operators (i.e., they commute with $\Gamma$)
$$
q\equiv q(u):=\sigma_3\otimes\left(\begin{array}{cc}
0 & -i u^* \\
i u & 0
\end{array} \right)\,,\quad 
\tilde\D := \sigma_2\otimes {\rm Id}_2\otimes \D, \quad
\tD_{t}:= (1-t)\tilde\D - tq\tilde\D q, \quad
\tD_{t,s} := \tD_{t} + sq\,.
$$
These unbounded operators are affiliated
with $\tilde{\mathcal N}$.
We begin by identifying $\B_1^\infty(\tilde\D,p)$.

\begin{lemma}
Let $\D$ be a self-adjoint operator affiliated with a semifinite von Neumann algebra $\cn$ endowed
with a semifinite normal faithful trace $\tau$. Then, with the notations introduced above, 
$
\B_1^\infty(\tilde\D,p) =M_4\big(\B_1^\infty(\D,p)\big)
$.
\end{lemma}

\begin{proof}
Note that
$
|\tilde\D|={\rm Id}_4\otimes|\D|\,,
$
so that the result follows from the definition of $\B_1^\infty(\D,p)$.
\end{proof}

We also let $q_e$ be the operator $ q(u)$ when $u={\rm Id}_{\A^\sim}$, that is,
$q_e=\s_3\otimes\s_2\otimes{\rm Id}_{\cn}$. 
In particular the Lemma above,  implies that
$$
q-q_e=\sigma_3\otimes\left(\begin{array}{cc}
0 & -i( u^*-{\rm Id}_{\cn}) \\
i( u-{\rm Id}_{\cn}) & 0
\end{array} \right),
$$
belongs to the algebra
$ M_4(\B_1^\infty(\D,p))=\B_1^\infty(\tilde\D,p) $
by the assumption of smooth summability (Definition \ref{delta-phi}) 
and since for a unitary $u\in\A^\sim$, using the 
definition after Equation \eqref{ext-A-sim}, we have ${\bf 1}_u={\rm Id}_{\A^\sim}$. Note also that
\begin{equation}
\label{ouf}
\{\tilde\D_\mu,q\}:=\tilde\D q + q\tilde\D =\sigma_1\otimes\left(\begin{array}{cc}
0 & [\D, u^*] \\
-[\D, u] & 0
\end{array} \right)\in M_4(\B_1^\infty(\D,p))\,.
\end{equation}

Set $\rho:=\s_2\otimes {\rm Id}_2\otimes {\rm Id}_{\cn}$,  so that 
$\rho$ anticommutes with $q$ and commutes with $\tilde\D$ and $\Gamma$.
Note that
$$ 
\tD_{t}\equiv \tD_{t,0}= \sigma_2\otimes \left(\begin{array}{cc}
\D+t u^*[\D, u] & 0 \\
0 & \D+tu[\D, u^*]
\end{array} \right)\in M_4({\rm OP}^1(\D)).
$$
Taking  derivatives in ${\rm OP}^1(\D)$, we get
$$
\frac{d\tD_{t}}{dt} = \sigma_2\otimes \left(\begin{array}{cc}
u^*[\D,u] & 0 \\
0 & u[\D,u^*]
\end{array} \right)\in\B_1^\infty(\tilde\D,p)\,.
$$

Define the graded trace on $\tilde{\mathcal N}$, by setting
$S\tau(A): =\frac{1}{2}\tilde\tau(\Gamma A)$, for $A$ of trace-class in
$\tilde{\mathcal N}$. 
For example, for $r>0$ we have
\begin{align*}
&\frac{d\tD_{t}}{dt}(1+\tD_{t}^2)^{-p/2-r}=\\
&\qquad \sigma_2\otimes \begin{pmatrix}
u^*[\D,u]\big(1+(\D+tu^*[\D,u])^2\big)^{-p/2-r} & 0 \\
0 & u[\D,u^*]\big(1+(\D+tu[\D,u^*])^2\big)^{-p/2-r}
\end{pmatrix},
\end{align*}
which is of trace-class on $\tilde\cn$ 
by \cite[Lemma 2.40]{CGRS2} since 
$u^*[\D,u]$ and $u[\D,u^*]$ belong to the algebra $\B_1^\infty(\D,p)$,
and since $\B_1^\infty(\D,p)=\B_1^\infty(\D+tu^*[\D,u],p)=\B_1^\infty(\D+tu[\D,u^*],p)$ by 
Proposition \ref{affine3}. These observations prove most of the next lemma.

\begin{lemma}
\label{mains}
With the notation above, and $r>0$, we have
\begin{align*}
&\int_0^1S\tau\Big(\frac{d\tD_{t}}{dt}(1+\tD_{t}^2)^{-p/2-r}\Big)dt\\
&\qquad\quad=
\int_0^1\tau\Big(u^*[\D,u]\big(1\!+\!(\D+tu^*[\D ,u])^2\big)^{-p/2-r}\!\!- 
u[\D,u^*]
\big(1\!+\!(\D+tu[\D,u^*])^2\big)^{-p/2-r}\Big)dt\\
&\qquad\quad=2\int_0^1\tau\Big(u^*[\D,u]\big(1\!+\!(\D+tu^*[\D ,u])^2\big)^{-p/2-r}\Big)dt.
\end{align*}
\end{lemma}

\begin{proof}
We only need to justify the last equality. The following elementary calculation does this.
\begin{align*}
&\int_0^1\tau\Big(u[\D,u^*]\big(1\!+\!(\D+tu[\D,u^*])^2\big)^{-p/2-r}\Big)dt\\
&=\int_0^1\tau\Big(u[\D,u^*]u\big(1\!+\!(u^*\D u+t[\D,u^*]u)^2\big)^{-p/2-r}u^*\Big)dt\\
&=\int_0^1\tau\Big(-[\D,u]\big(1\!+\!(u^*\D u-tu^*[\D,u])^2\big)^{-p/2-r}u^*\Big)dt\\
&=\int_0^1\tau\Big(-u^*[\D,u]\big(1\!+\!((1-t)u^*\D u+t\D)^2\big)^{-p/2-r}\Big)dt\\
&=\int_1^0\tau\Big(u^*[\D,u]\big(1\!+\!(wu^*\D u+(1-w)\D)^2\big)^{-p/2-r}\Big)dw,\qquad w=1-t\\
&=-\int_0^1\tau\Big(u^*[\D,u]\big(1\!+\!(\D+wu^*[\D,u])^2\big)^{-p/2-r}\Big)dw.\qedhere
\end{align*}
\end{proof}

\subsection{Obtaining a preliminary formula from the resolvent cocycle}
Our plan is to reverse the argument in \cite{CPRS2}. This means we 
plan to go from the resolvent cocycle to a spectral flow formula.
First we calculate
\begin{equation}
\label{interesting}
\tD_{t,s}^2= \tD_{t}^2+ s(1-2t)\{\tD,q\} +s^2\,.
\end{equation}
We prove a trace class result for this family of operators.

\begin{lemma}
\label{lem:traceable}
With the notations above, we have
$$
(q-q_e)\big(1+\tD^2+s^2+s\{\tD,q\}\big)^{-p/2-r}\in \L^1(\tilde\cn,\tilde\tau)\,,
\quad\forall \,r\in\C\quad\mbox{with}\quad \Re(r)>0.
$$
\label{cvb}
\end{lemma}
\begin{proof}
As seen earlier, $ q-q_e\in\B_1^\infty(\tD,p)$. 
Set $\tilde\delta:=[(1+\tD^2)^{1/2},\cdot]$ and 
$\delta:=[(1+\D^2)^{1/2},\cdot]$. Then we get for all $n\in\N$
$$
\tilde\delta^n(q)=\sigma_3\otimes\begin{pmatrix}
0&-i\delta^n(u^*)\\-i\delta^n(u)&0
\end{pmatrix}\,.
$$
Thus $q$ belongs to the intersection of the domains of the 
powers of the derivation $\tilde\delta$, so that we can 
apply Proposition \ref{affine1}, which in this context gives
 $\B_1(\tD,p)=\B_1(\tD+sq,p)$. The proof is completed 
 by using $q^2=1$ so that one has
 $\tD^2+s^2+s\{\tD,q\}=(\tD+sq)^2$.
\end{proof}

\begin{lemma}\label{reso}
With the notation as above, and with
$\Re(r)>0$ there exists $\delta\in(0,1)$ such that  with $M=2\lfloor p/2\rfloor+1$:
\begin{align*}
&\int_0^\infty S\tau
\Big(q\big(1+\tD^2+s^2+s\{\tD,q\}\big)^{-p/2-r}-q_e\big(1+\tD^2+s^2\big)^{-p/2-r}\Big)ds\\
&\qquad\qquad= \frac{1}{2\pi i}\sum_{m=1,odd}^{M}\int_0^\infty s^m S\tau\left(\int_\ell
\lambda^{-p/2-r}q\left(R_s(\lambda)\{\tD,q\}\right)^mR_s(\lambda)d\lambda\right)ds
+{\rm holo},
\end{align*}
where $\rm  holo$ is a function of $r$ holomorphic for $\Re(r)>-p/2+\delta$.
\end{lemma}

\begin{proof}
We use Cauchy's formula to write
\begin{align*}
&q\big(1+\tD^2+s^2+s\{\tD,q\}\big)^{-p/2-r}=
\frac{1}{2\pi i}\int_\ell
\lambda^{-p/2-r}q\big(\lambda-(1+\tilde\D^2+s\{\tD,q\}+s^2)\big)^{-1}d\lambda\,,
\end{align*}
where $\ell$ is the vertical line $\ell=\{a+iv:v\in\R\}$ with $0<a<1/2$.
Then we apply the resolvent expansion (as in Section 7 of \cite{CPRS2})  to arrive at
\begin{align}
\label{sum}
q\big(1+\tD^2+s^2+s\{\tD,q\}\big)^{-p/2-r}
&=\frac{1}{2\pi i}\int_\ell\lambda^{-p/2-r}\sum_{m=0}^{M}s^mq\left(R_s(\lambda)
\{\tD,q\}\right)^mR_s(\lambda)d\lambda\\
&\quad+s^{M+1}\frac{1}{2\pi i}\int_\ell\lambda^{-p/2-r}q
\left(R_s(\lambda)\{\tD,q\}\right)^{M+1}
\tilde{R}_s(\lambda)d\lambda\,,\nonumber
\end{align}
where $M=2\lfloor p/2\rfloor+1$ and
we use the notations
$$
R_s(\lambda)=\big(\lambda-(1+s^2+\tD^2)\big)^{-1}\quad\mbox{and} 
\quad\tilde R_s\big(\lambda)=(\lambda-(1+s^2+s\{\tilde\D,q\}+\tilde\D^2)\big)^{-1}.
$$

By \cite[Lemma 4.3]{CGRS2} and  Equation \eqref{ouf}, 
we see that  the terms with $m=1,\dots,M$ are trace-class 
for $\Re(r)>0$. By \cite[Lemma 2.42]{CGRS2}, so is the remainder term. 
Thus, the only term in this expansion which is not trace-class is the term with 
$m=0$, namely
$$
\frac{1}{2\pi
i}\int_\ell\lambda^{-p/2-r}q\,R_s(\lambda)\,d\lambda=q\big(1+\tilde\D^2+s^2\big)^{-p/2-r}\,.
$$
However, $(q-q_e)\big(1+\tilde\D^2+s^2\big)^{-p/2-r}$ is 
trace class and it has a vanishing super-trace.
Indeed, since $\rho^2={\rm Id}_{\tilde\cn}$ and that
$\rho$ commutes with $\tD$ and $\Gamma$, but 
anticommutes with $q$ and $q_e$, we find
\begin{align*}
S\tau\Big((q-q_e)\big(1+\tilde\D^2+s^2\big)^{-p/2-r}\Big)
&=\tilde\tau\Big(\Gamma\rho^2(q-q_e)\big(1+\tilde\D^2+s^2\big)^{-p/2-r}\Big)\\
&=-\tilde\tau\Big(\Gamma\rho(q-q_e)\big(1+\tilde\D^2+s^2\big)^{-p/2-r}\rho\Big)\\
&=-S\tau\Big((q-q_e)\big(1+\tilde\D^2+s^2\big)^{-p/2-r}\Big)\,.
\end{align*}

Similarly, if we consider a single term in the sum \eqref{sum}, with
$m>0$ we find 
\begin{align*}
&S\tau\left(\frac{1}{2\pi i}\int_\ell\lambda^{-p/2-r}q\left(R_s(\lambda)\{\tD,q\}\right)^m
R_s(\lambda)d\lambda\right)\\
&\qquad\qquad\qquad\qquad\qquad\qquad\qquad=(-1)^{m+1}S\tau\left(\frac{1}{2\pi
i}\int_\ell\lambda^{-p/2-r}
q\left(R_s(\lambda)\{\tD,q\}\right)^kR_s(\lambda)d\lambda\right).
\end{align*}
So if $m$ is even we get zero. 
This argument does not apply to the remainder
term
\ben 
s^{M+1}S\tau\left(\frac{1}{2\pi
i}\int_\ell\lambda^{-p/2-r}(q-q_e)
\left(R_s(\lambda)\{\tD,q-q_e\}\right)^{M+1}
\tilde{R}_s(\lambda)d\lambda\right),
\een
as $\rho$ neither commutes nor anticommutes with $\tilde{R}_s(\lambda)$.
However, the integral over $s$ of this 
remainder term is holomorphic at $r=(1-p)/2$, by \cite[Lemma 2.42]{CGRS2}.
Integrating  the remaining terms
over $s\in[0,\infty]$ using  \cite[Lemma 4.16]{CGRS2} yields the result.
\end{proof}

Next we need to relate this expression above to the resolvent cocycle evaluated on the
Chern character ${\rm Ch}(u)$.
Following \cite[section 7]{CPRS2}, we get
\begin{align*} &q\left(R_s(\lambda)\{\tD,q\}\right)^mR_s(\lambda)\\
&=i(-1)^{(m-1)/2}\s_3\s_1^m\otimes\begin{pmatrix} u^*R[\D,u]R[\D,u^*]\cdots
[\D,u]R & 0\\ 0 & uR[\D,u^*]R[\D,u]\cdots
[\D ,u^*]R\end{pmatrix}.
\end{align*}
On the right hand side we have written
$R\equiv (\lambda-(1+s^2+\D^2))^{-1}$ for the resolvent of $\D^2$.
Recall that the grading operator is $\Gamma=\s_2\otimes\sigma_3\otimes{\rm Id}_{M_2(\cn)}$,
and that $\s_2\s_3\s_1^m=i{\rm Id}_2$ for $m$ odd. Writing $\tr_4$ for the
operator-valued trace which maps
$\tilde{\mathcal N}= M_4(\mathcal N)\to \mathcal N$, we have
\begin{align*}
&\tr_4\Big(\Gamma qR_s(\lambda)\{\tD,q\}R_s(\lambda)\cdots\{\tD,q\}
R_s(\lambda)\Big)\\
&\qquad\quad\quad=2(-1)^{(m+1)/2}\big(u^*R[\D,u]R[\D,u^*]
\cdots [\D,u]R-uR[\D,u^*]R[\D,u]
\cdots [\D,u^*]R\big).
\end{align*}
Consequently, there is a $\delta$ with $0<\delta<1$ such that for $\Re(r)>0$
\begin{align*}
&\int_0^\infty
S\tau\Big(q\big(1+\tD^2+s^2+s\{\tD,q\}\big)^{-p/2-r}-q_e(1+\tD^2+s^2\big)^{-p/2-r}\Big)ds\\
&\quad=\frac{1}{2\pi i}\sum_{m=1,odd}^{M}(-1)^{(m+1)/2}\int_0^\infty s^m
\tau\Big(\int_\ell \lambda^{-p/2-r}
\Big(u^*R[\D,u]R[\D,u^*]\cdots [\D,u]R\\
&\qquad\qquad\qquad\qquad\qquad\qquad\qquad\qquad- uR
[\D,u^*]R[\D,u]\cdots [\D,u^*]R\Big)d\lambda\Big)ds
+{\rm holo}\\
&=\frac{1}{\sqrt{2\pi i}}\frac{1}{2}\sum_{m=1,odd}^M\phi^r_m({\rm Ch}_m(u)-{\rm Ch}_m(u^*))+{\rm holo}\,,
\end{align*}
where ${\rm holo}$ is a function of $r$ holomorphic for $\Re(r)>-p/2+\delta/2$, and the
last line just comes from comparing constants in the definition of $\phi^r_m$ and ${\rm Ch}_m$.

The following integral formula for the spectral flow  now follows directly from Theorem \ref{thm:sf=cocycle}. 
This is the main intermediate step to the proof of Theorem \ref{specflow}.

\begin{prop}\label{resolvent_11}
Let $M=2\lfloor p/2 \rfloor +1$. Then
\begin{align*} 
sf(\D_\mu,\hat{u}^*\D_\mu \hat{u}) &=
{\rm Res}_{r=(1-p)/2}\int_0^\infty \!\!\!\!\!S\tau\Big(q\big(1+\tD^2+s^2+s\{\tD,q\}\big)^{-p/2-r}
\!\!\!\!\!\!-q_e \big(1+\tD^2+s^2\big)^{-p/2-r}\Big)ds.
\end{align*}
\end{prop}

{\bf Remark.}  Since $q_e$ anticommutes with $\tD$, the formula above may also be written as
\begin{align*} 
sf(\D_\mu,\hat{u}^*\D_\mu \hat{u})&=
{\rm Res}_{r=(1-p)/2}\int_0^\infty S\tau\Big(q\big(1+(\tD+sq)^2\big)^{-p/2-r}
-q_e \big(1+(\tD+sq_e)^2\big)^{-p/2-r}\Big)ds\,.
\end{align*}

%
%
%
%
%

\subsection{Exact one forms}

Proposition \ref{affine3} shows that if $\D$ is unbounded and 
self-adjoint, then the space
$
\D+{\rm OP}^0(\D)_{\rm sa} = \D+\B_1^\infty(\D,p)\,,
$
is a real affine Fr\'echet  space whose topology is  independent of the base point.

\begin{definition}
\label{VG}
Let $\Phi$ be the two-dimensional real affine space
$$
\Phi:=\big\{\tilde\D+X\,:\,X=\alpha q\{\tD,q\}+\beta q\,,\quad\alpha,\beta\in\R\big\}\,.
$$
For $X=\alpha q\{\tD,q\}+\beta q$,  set ${\bf 1}_X:=\beta q_e$. 
(This is consistent with the earlier notation of Equation \eqref{ext-A-sim}.)
We then consider the one form, 
\begin{align}
\label{une}
X\in T_{\tD+Y}\Phi\mapsto 
\tilde\tau\Big(X(1+(\tD+Y)^2)^{-p/2-r}-{\bf 1}_X(1+(\tD+{\bf1}_Y)^2)^{-p/2-r}\Big)\,,\quad\Re(r)>0\,,
\end{align}
defined on the tangent space of $\Phi$
at $\tD+Y$.
\end{definition}

Our strategy in this subsection is to prove that the one form of Equation \eqref{une} is
well-defined, differentiable in trace norm, and closed. Since $\Phi$ is an affine space, a 
Poincar\'{e} Lemma argument then shows that the one form is exact.


\begin{lemma}
\label{gros}
For any $r\in\C$ with $\Re(r)>0$, the map \eqref{une} is well defined.
\end{lemma}

\begin{proof}
Let $r\in\C$ with $\Re(r)>0$.
First write 
\begin{align*}
&X(1+(\tD+Y)^2)^{-p/2-r}-{\bf 1}_X(1+(\tD+{\bf 1}_Y)^2)^{-p/2-r}\\
&\quad=(X-{\bf 1}_X)(1+(\tD+Y)^2)^{-p/2-r}-{\bf 1}_X\Big((1+(\tD+{\bf1}_Y)^2)^{-p/2-r}
-(1+(\tD+Y)^2)^{-p/2-r}\Big)\,.
\end{align*}
The first term is trace-class since $\Re(r)>0$, $X-{\bf 1}_X\in \B^\infty_1(\tD,p)$ 
and $\B^\infty_1(\tD+Y,p)=\B^\infty_1(\tD,p)$ by Proposition \ref{affine3}. For the second term, we employ the 
Laplace transform representation and Duhamel formula, yielding
\begin{align*}
&(1+(\tD+Y)^2)^{-p/2-r}-(1+(\tD+{\bf1}_Y)^2)^{-p/2-r}\\
&\quad\qquad\qquad\qquad\qquad\qquad
=\frac{1}{\Gamma(p/2+r)}
\int_0^\infty t^{p/2+r-1} e^{-t}\Big(e^{-t(\tD+Y)^2}-e^{-t(\tD+{\bf1}_Y)^2}\Big)dt\\
&\quad\qquad\qquad\qquad\qquad\qquad
=-\frac{1}{\Gamma(p/2+r)}\int_0^\infty 
t^{p/2+r} e^{-t}\Big(\int_0^1e^{-st(\tD+Y)^2} \,Z \,e^{-(1-s)t(\tD+{\bf1}_Y)^2}ds\Big)dt \,,
\end{align*}
where we have set 
$$
Z:=(\tD+Y)^2-(\tD+{\bf1}_Y)^2=(Y-{\bf1}_Y)\tD+\tD
(Y-{\bf1}_Y)+Y^2-{\bf1}_Y^2\,.
$$
By assumption $Y-{\bf1}_Y$ belongs to $\B_1^\infty(\tD,p)$ and 
a short computation shows that
$Y^2-{\bf1}_Y^2$ belongs to the same space.  Let us estimate the 
trace-norm of the operator corresponding to the
first term in $Z$. First for $s\in [0,1/2]$, we have
\begin{align*}
\big\|e^{-st(\tD+Y)^2} \,(Y-{\bf1}_Y)\tD \,e^{-(1-s)t(\tD+{\bf1}_Y)^2}\big\|_1
\leq&\|(Y-{\bf1}_Y)(1+(\tD-{\bf1}_Y)^2)^{-p/2-\Re(r)}\|_1\\
&\times\|(1+(\tD-{\bf1}_Y)^2)^{p/2+\Re(r)} \tD e^{-t(\tD+{\bf1}_Y)^2/2}\|\,.
\end{align*}
Since ${\bf 1}_Y$ is proportional to $q_e$, it anticommutes with $\tD$ and thus 
$$
\tD\,e^{-t(\tD+{\bf1}_Y)^2/2}=e^{-t(\tD-{\bf1}_Y)^2/2}\,\tD\,.
$$
Hence we obtain the norm estimate
\begin{align*}
&\|(1+(\tD-{\bf1}_Y)^2)^{p/2+\Re(r)} \tD e^{-t(\tD+{\bf1}_Y)^2/2}\|
=\|(1+(\tD-{\bf1}_Y)^2)^{p/2+\Re(r)}  e^{-t(\tD-{\bf1}_Y)^2/2}\tD\|\\
&\leq \|(1+(\tD-{\bf1}_Y)^2)^{p/2+\Re(r)}  e^{-t(\tD-{\bf1}_Y)^2/2}(\tD-{\bf1}_Y)\|\\
&\qquad\qquad\qquad\qquad
+\|{\bf1}_Y\| \|(1+(\tD-{\bf1}_Y)^2)^{p/2+\Re(r)}  e^{-t(\tD-{\bf1}_Y)^2/2}\|\\
&\qquad\qquad\qquad\qquad\qquad\qquad\leq c_1 t^{-p/2-\Re(r)-1/2}+c_2 t^{-p/2-\Re(r)},
\end{align*}
by elementary spectral theory.
 For $s\in [1/2,1]$, we have
\begin{align*}
\big\|e^{-st(\tD+Y)^2} \,(Y-{\bf1}_Y)\tD \,e^{-(1-s)t(\tD+{\bf1}_Y)^2}\big\|_1
&\leq\big\|e^{-t(\tD+Y)^2/2} \big(1+(\tD+Y)^2\big)^{p/r+\Re(r)+1/2}\|\\
&\qquad\times\|\big(1+(\tD+Y)^2\big)^{-p/r-\Re(r)-1/2}(Y-{\bf1}_Y)\tD \|_1\\
&\leq c_3 t^{-p/2-\Re(r)-1/2}\,,
\end{align*}
where we have estimated the last trace norm by a constant (depending on $r$) 
using \cite[Lemma 2.42]{CGRS2}. 
The operator corresponding to the 
second term in $Z$ gives the same contribution. 
Indeed,  as a change of variable under the $s$-integral shows, it is
the adjoint of the first term. The third and last term in $Z$ is  
even more easily estimated in trace norm, again using \cite[Lemma 2.42]{CGRS2}, 
by  $c_4 t^{-p/2-\Re(r)}$. Thus, we get
\begin{align*}
\big\|(1+(\tD+Y)^2)^{-p/2-r}-(1+(\tD+{\bf1}_Y)^2)^{-p/2-r}\big\|_1\leq 
\frac{1}{\Gamma(p/2+r)}\int_0^\infty (C_1t^{-1/2}+C_2) e^{-t}dt\,,
\end{align*}
and the proof is complete.
\end{proof}

\begin{rem*}
{\rm The following result  establishing that the one form is closed may be 
generalised using \cite{Ge}, however we adopt a different approach sufficient for our purposes.}
\end{rem*}

To prove that the one form is closed, we must establish that
\begin{align}
\label{closed}
&\frac{d}{dt}\Big\vert_{t=0}\tilde\tau\Big(X(1+(\tD+t Y)^2)^{-p/2-r}
-{\bf1}_X(1+(\tD+t{\bf1}_{Y})^2)^{-p/2-r}\Big)\\
&\quad=\frac{d}{dt}\Big\vert_{t=0}\tilde\tau\Big(Y(1+(\tD+tX)^2)^{-p/2-r}
-{\bf1}_Y(1+(\tD+t{\bf1}_{X})^2)^{-p/2-r}\Big)\,,
\quad\forall\, \Re(r)>0,\nonumber\
\end{align}
and for all $X,Y$ in the tangent space of $\Phi$ at 
$\tD$. The proof of this fact is a corollary of the following result.

\begin{prop}
\label{prop:use-again}
Let $X,Y\in T_{\tD}\Phi$. Then for all $\Re(r)>0$, the map
$$
\R\ni t\mapsto \alpha_{X,Y}^r(t)
:= Y(1+(\tD+tX)^2)^{-p/2-r}-{\bf 1}_Y(1+(\tD+t{\bf1}_{X})^2)^{-p/2-r}\in\tilde\cn\,,
$$
is  differentiable at $t=0$ in the trace-norm topology. 
Moreover, the value of its  derivative  at $t=0$ is given by 
$$
{\dot\alpha_{X,Y}^r}(0)
=-\frac1{\Gamma(p/2+r)}\int_0^\infty 
u^{p/2+r}e^{-u}\int_0^1Y\,e^{-su\tD^2}\,\{\tD,X\}\,e^{-(1-s)u\tD^2}\,ds\,du\,.
$$
\label{unedeplus}
\end{prop}

\begin{proof}
We use the same integral representation as in  Lemma \ref{gros}, 
and the fact that $\{\tD,{\bf 1}_X\}=0$, to get
\begin{align*}
&\alpha_{X,Y}^r(t)=Y(1+(\tD+tX)^2)^{-p/2-r}-{\bf 1}_Y(1+(\tD+t{\bf1}_{X})^2)^{-p/2-r}\\
&=\frac{1}{\Gamma(p/2+r)}\int_0^\infty u^{p/2+r-1} e^{-u}\Big(
Ye^{-u(\tD+tY)^2}-
{\bf 1}_Ye^{-u(\tD+t{\bf1}_Y)^2}\Big)du\\
&=\frac{1}{\Gamma(r+p/2)}\int_0^\infty u^{p/2+r} e^{-u}\int_0^1\Big(
(Y-{\bf 1}_Y)e^{-u\tD^2}-Ye^{-su(\tD+tX)^2}\big(t\{X,\tD\}+t^2X^2\big)
e^{-su\tD^2} \\
&\qquad\qquad\qquad\qquad\qquad\qquad\qquad\qquad\qquad\qquad\qquad\qquad
-{\bf 1}_Ye^{-su(\tD+t{\bf 1}_X)^2}t^2X^2e^{-su\tD^2}ds\Big)du\,.
\end{align*}
Hence, 
\begin{align*}
&\frac{\alpha_{X,Y}^r(t)-\alpha_{X,Y}^r(0)}t-\dot\alpha^r_{X,Y}(0)=\\
&\frac{1}{\Gamma(r+p/2)}\int_0^\infty u^{p/2+r} e^{-u}\int_0^1\Big(
-Y\big(e^{-su(\tD+tX)^2}-e^{-su\tD^2}\big)\{\tD,X\}e^{-su\tD^2}ds\Big)du\\
&-\frac{t}{\Gamma(r+p/2)}\int_0^\infty u^{p/2+r} e^{-u}\int_0^1\Big(
Ye^{-su(\tD+tX)^2}X^2e^{-su(\tD)^2}  +{\bf 1}_Ye^{-su(\tD+t{\bf 1}_X)^2}X^2
e^{-su\tD^2}ds\Big)du.
\end{align*}
To estimate the trace norm of the second line on the right hand side, 
one can use exactly the same method as in the proof of Lemma \ref{gros}. 
This gives a bound of the form
$$
C\,(\|Y\|+1)\|X\|\|X(1+\tD^2)^{-p/2-\Re(r)}\|_1\,t\,,
$$
for some positive constant $C$. Thus the corresponding term 
goes to zero with $t$. To estimate the first line, we need one more 
application of the Duhamel formula, to obtain, up to a numerical factor,
\begin{align*}
t\,\int_0^\infty u^{p/2+r} e^{-u}\int_0^1\int_0^1Y e^{-vsu(\tD+tX)^2}\,\big(\{\tD,X\}+tX^2\big)\,
e^{-(1-v)su\tD^2}\,\{\tD,X\}\,e^{-su\tD^2}dvdsdu\,.
\end{align*}
Written in this form, we can now use the same method as in Lemma \ref{gros} to 
conclude that this term goes to zero in trace norm with $t$ also.
 \end{proof}

 We may now prove that our one form is exact.
 
 \begin{corollary}
 \label{cor:exact}
 Let $X,Y\in T_{\tD}\Phi$. Then for all $\Re(r)>0$, we have
\begin{equation}
\label{mostimportant}
 \tilde\tau\big({\dot\alpha_{X,Y}^r}(0)\big)= \tilde\tau\big({\dot\alpha_{Y,X}^r}(0)\big)\,.
 \end{equation}
\end{corollary}

\begin{proof}
From  Proposition \ref{unedeplus}, we know that in trace-norm topology
$$
{\dot\alpha_{X,Y}^r}(0)
=-\frac1{\Gamma(r+p/2)}\int_0^\infty u^{p/2+r}e^{-u}\int_0^1
Y\,e^{-su\tD^2}\,\{\tD,X\}\,e^{-(1-s)u\tD^2}\,ds\,du\,.
$$
Since 
$$
\big\|Y\,e^{-su\tD^2}\,\{\tD,X\}\,e^{-(1-s)u\tD^2}\big\|_1
\leq c_1 t^{-p/2-\Re(r)-1/2}+ c_2 t^{-p/2-\Re(r)}\,,
$$
as shown by elementary estimates as in the proof of Proposition \ref{prop:use-again}, we deduce
that the integral representation above for 
${\dot\alpha_{X,Y}^r}(0)$ is absolutely convergent in trace-norm.
Thus, we may apply the Fubini Theorem to get
\begin{align*}
 \tilde\tau\big({\dot\alpha_{X,Y}^r}(0)\big)
 &=-\frac1{\Gamma(p/2+r)}\int_0^\infty u^{p/2+r}e^{-u}
 \int_0^1\tilde\tau\big(Y\,e^{-su\tD^2}\,\{\tD,X\}\,e^{-(1-s)u\tD^2}\big)\,ds\,du\\
 &=-\frac1{\Gamma(p/2+r)}\int_0^\infty u^{p/2+r}e^{-u}
 \int_0^1\tilde\tau\big(X\,e^{-(1-s)u\tD^2}\,\{\tD,Y\}\,e^{-su\tD^2}\big)\,ds\,du\\
  &=-\frac1{\Gamma(p/2+r)}\int_0^\infty \!\!\!u^{p/2+r}e^{-u}
  \int_0^1\!\!\!\tilde\tau\big(X\,e^{-su\tD^2}\,\{\tD,Y\}\,e^{-(1-s)u\tD^2}\big)\,ds\,du=\tilde\tau\big({\dot\alpha_{Y,X}^r}(0)\big),
\end{align*}
where we used  the cylicity of the trace, see \cite{BrK}, 
to get the second equality, and the change of variable $s\mapsto 1-s$ to get the third.
\end{proof}

The corollary establishes that our one form is closed, and as $\Phi$ is an affine space, 
a Poincar\'{e} Lemma style argument then shows that our one form is exact.


\subsection{A new spectral flow formula}
We now use the exactness of our one form to change the integration path,
and obtain a formula similar to those in \cite{CP1,CP2}. We start with the following observation.

\begin{lemma}\label{horizint} Let $(\A,\HH,\D)$
be a nonunital smoothly summable spectral triple of spectral dimension
$p\geq 1$. Then, for any  $s_0 > 0$ and any $\Re(r)>0$,
\begin{align}
\label{ch}
&\int_0^{s_0}
S\tau\left(\frac{d\tD_{1,s}}{ds}(1+\tD_{1,s}^2)^{-p/2-r}-q_e 
\big(1+\tD^2+s^2\big)^{-p/2-r}\right)ds\\
&\qquad\qquad\qquad\quad=-\int_0^{s_0}
S\tau\left(\frac{d\tD_{0,s}}{ds}(1+\tD_{0,s}^2)^{-p/2-r}-q_e 
\big(1+\tD^2+s^2\big)^{-p/2-r}\right)ds.
\nonumber
\end{align}
\end{lemma}

\begin{proof}
Note that since
  $\tD_{1,s}=-q\tD q+sq$ and $\tD_{0,s}=\tD +sq$,
we have
$
\frac{d\tD_{1,s}}{ds}=q=\frac{d\tD_{0,s}}{ds}.
$
Now observe that if $X=q$ and $Y=-q\{\tD,q\}+sq$, we have 
(with the notation of  Definition \ref{VG})
${\bf 1}_X=q_e$ and ${\bf 1}_Y=sq_e$. Thus,
\begin{align*}
&\frac{d\tD_{1,s}}{ds}(1+\D_{1,s}^2)^{-p/2-r}-q_e \big(1+\tD^2+s^2\big)^{-p/2-r}\\
&\qquad\qquad\qquad\qquad\qquad\qquad=
X(1+(\tD+Y)^2)^{-p/2-r}-{\bf 1}_X \big(1+(\tD+{\bf 1}_Y)^2\big)^{-p/2-r}\,,
\end{align*}
which is trace class by Lemma \ref{gros}. Similar comments 
apply to the second line in \eqref{ch}. Thus, each side of the equality \eqref{ch} is well defined.

Recall that $\rho=\s_2\otimes {\rm Id}_2\otimes {\rm Id}_2$, so that
$\rho q\rho=-q$, $\rho q_e\rho=-q_e$, $\rho^2=1$ and $\rho\Gamma\rho=\Gamma$. 
Since also $\rho\tD=\tD\rho$,  one easily
calculates that
$$
\rho q\tD_{1,s}=-(\tD +sq)\rho q = -\tD_{0,s}\rho q,
$$
so that
$
\rho q\tD_{1,s}^2=\tD_{0,s}^2\rho q,
$
and hence for any Borel function, $f$, we have
$
\rho q f(\tD_{1,s}^2)=f(\tD_{0,s}^2)\rho q.
$
Also, since $q_e$ anticommutes with $\tD$, it commutes with  $\tD^2$, so that
$$
\rho q_e \big(1+\tD^2+s^2\big)^{-p/2-r}= \big(1+\tD^2+s^2\big)^{-p/2-r}\rho q_e .
$$
Then
\begin{align*}
&2S\tau\Big(\frac{d\tD_{1,s}}{ds}(1+\tD_{1,s}^2)^{-p/2-r}-
q_e \big(1+\tD^2+s^2\big)^{-p/2-r}\Big)\\
&\qquad\qquad\qquad\qquad\qquad\qquad=
\tilde\tau\left(\Gamma \rho^2q(1+\tD_{1,s}^2)^{-p/2-r}-
\Gamma \rho^2q_e \big(1+\tD^2+s^2\big)^{-p/2-r}\right)\\
&\qquad\qquad\qquad\qquad\qquad\qquad= 
\tilde\tau\left(\Gamma \rho (1+\tD_{0,s}^2)^{-p/2-r}\rho q-
\Gamma \rho  \big(1+\tD^2+s^2\big)^{-p/2-r}\rho q_e\right)\\
&\qquad\qquad\qquad\qquad\qquad\qquad= 
\tilde\tau\left(\rho q\Gamma \rho (1+\tD_{0,s}^2)^{-p/2-r} -
\rho q_e\Gamma \rho  \big(1+\tD^2+s^2\big)^{-p/2-r}\right)\\
&\qquad\qquad\qquad\qquad\qquad\qquad= 
-\tilde\tau\left(\rho \Gamma \rho q (1+\tD_{0,s}^2)^{-p/2-r} -
\rho \Gamma \rho q_e \big(1+\tD^2+s^2\big)^{-p/2-r}\right)\\
&\qquad\qquad\qquad\qquad\qquad\qquad
=-2S\tau\Big(\frac{d\tD_{0,s}}{ds}(1+\tD_{1,s}^2)^{-p/2-r}-
q_e \big(1+\tD^2+s^2\big)^{-p/2-r}\Big)\,.
\end{align*}
This completes the proof.
\end{proof}

Integrating  our one-form (Definition \ref{VG})   around the boundary of the closed rectangle
$$
\{\tD_{t,s},\,(t,s)\in[0,1]\times[0,s_0]\}\subset\Phi\,,
$$
of our two-dimensional affine space $\Phi$ gives zero
by exactness (Corollary \ref{cor:exact}), and so
\begin{align*} 
&\int_0^{s_0}
S\tau\left(\frac{d\tD_{1,s}}{ds}(1+\tD_{1,s}^2)^{-p/2-r}-q_e \big(1+\tD^2+s^2\big)^{-p/2-r}\right)ds\\&-
\int_0^{s_0} S\tau\left(\frac{d\tD_{0,s}}{ds}(1+\tD_{0,s}^2)^{-p/2-r}-q_e \big(1+\tD^2+s^2\big)^{-p/2-r}\right)ds\\
&=-\int_0^1S\tau\left(\frac{d\tD_{t,0}}{dt}(1+\tD_{t,0}^2)^{-p/2-r}\right)dt+
\int_0^1S\tau\left(\frac{d\tD_{t,s_0}}{dt}(1+\tD_{t,s_0}^2)^{-p/2-r}\right)dt\,.
\end{align*}
Note that there are no ``extra terms'' coming from $q_e$ 
on the right hand side of the preceding equality because
if 
$$
X:=\frac{d\tD_{t,s}}{dt}=-q\{\tD,q\}\in\B_1^\infty(\tD,p)\,,
$$
one has  ${\bf 1}_X=0$.
Rearranging, using Lemma \ref{horizint} to combine the first two integrals 
gives

\begin{align}
\label{Q2Q}  
&2
\int_0^{s_0} S\tau\left(\frac{d\tD_{0,s}}{ds}(1+\tD_{0,s}^2)^{-p/2-r}
-q_e \big(1+\tD^2+s^2\big)^{-p/2-r}\right)ds\nonumber\\
&\qquad=\int_0^1S\tau\left(\frac{d\tD_{t,0}}{dt}(1+\tD_{t,0}^2)^{-p/2-r}\right)dt-
\int_0^1S\tau\left(\frac{d\tD_{t,s_0}}{dt}(1+\tD_{t,s_0}^2)^{-p/2-r}\right)dt\,.
\end{align}

To finish the argument we have to establish the next result.

\begin{prop}\label{prop3} Let $(\A,\HH,\D)$ be an odd
nonunital smoothly summable  spectral triple of spectral dimension
$p\geq 1$. Then, with the notations displayed above, we have
\ben
\lim_{s\to\infty}\int_0^1\Big\|\frac{d\tD_{t,s}}{dt}
(1+\tD_{t,s}^2)^{-p/2-r}\Big\|_1dt=0.
\een
\end{prop}

\begin{proof}
Remember that
$$
\frac{d\tD_{t,s}}{dt}=-q\{\tD,q\}\in\B_1^\infty(\tD,p).
$$
Then, as noted in Equation \ref{interesting}, we also have
$$
\tD_{t,s}^2=\tD_{t}^2+s(1-2t)\{\tD,q\}+s^2\,.
$$
Let $s\geq 2\|\{\tD,q\}\|$. For $t\in[0,1]$, we then we have the operator inequality
$$
\tD_{t}^2+s(1-2t)\{\tD,q\}+s^2
\geq \tD_{t}^2-s|1-2t|\|\{\tD,q\}\|+s^2\geq \tD_{t}^2+\tfrac12 s^2\,.
$$
This leads to
$$
\|(1+\tD_{t,s}^2)^{-\delta}\|\leq \|(1+\tD_{t}^2+\tfrac12 s^2)^{-\delta}\|
\leq \left(\frac {s^2}2\right)^{-\delta}\,,\quad \forall \delta>0\,.
$$
Now, let us fix $\delta>0$ such that $\Re(r)-\delta>0$. 
We then obtain
\begin{align*}
\Big\|\frac{d\tD_{t,s}}{dt}
(1+\tD_{t,s}^2)^{-r-p/2}\Big\|_1
&=\big\| q\{\tD,q\}\big(1+\tD_{t}^2+s(1-2t)\{\tD,q\}+s^2\big)^{-p/2-r}\big\|_1\\
&\leq  \left(\frac {s^2}2\right)^{-\delta}\,\|\{\tD,q\}
\big(1+\tD_{t}^2+s(1-2t)\{\tD,q\}+s^2\big)^{-p/2-r+\delta}\big\|_1\\
&\leq  C\left(\frac {s^2}2\right)^{-\delta}\,,
\end{align*}
where the constant
$$
C:= \|\{\tD,q\}\big(1+\tD_{t}^2\big)^{-p/2-r+\delta}\|_1
\|\big(1+\tD_{t}^2\big)^{p/2+r-\delta}\big(1+\tD_{t}^2
+s(1-2t)\{\tD,q\}+s^2\big)^{-p/2-r+\delta}\big\|_\infty\,,
$$
is a bounded function of $s\geq 2\|\{\tD,q\}\|$ by Lemma \ref{lem4}. This is enough to conclude.
\end{proof}

From  Proposition \ref{prop3}, we can let $s_0\to\infty$ in Equation \eqref{Q2Q}, to give explicitly 
\begin{align*}
&
\int_0^{\infty} S\tau\left(q(1+\tD^2+s\{\tD,q\}+s^2)^{-p/2-r}-
q_e \big(1+\tD^2+s^2\big)^{-p/2-r}\right)ds
\\
&\qquad\qquad\qquad\qquad\qquad\qquad\qquad\qquad\qquad\qquad\qquad
=\frac{1}{2}\int_0^1S\tau\left(\frac{d\tD_{t,0}}{dt}(1+\tD_{t,0}^2)^{-p/2-r}
\right)dt\,.
\end{align*}

By Proposition \ref{resolvent_11}, the residue of the
left hand-side gives the numerical index, and by Lemma \ref{mains}
the right hand side is 
$$
\int_0^1\tau\Big(u^*[\D,u]\big(1\!+\!(\D+tu^*[\D ,u])^2\big)^{-p/2-r}\Big)dt.
$$
This latter formula is formally similar to the  spectral flow formulas in \cite{CP1,CP2} for
the special case of unitarily equivalent endpoints, but the hypotheses we have used to derive it 
are very different from those in these earlier papers. This completes the proof of Theorem \ref{specflow}.

%
%
%
%
%
%


\begin{thebibliography}{APS3}

%
%

\bibitem{APS3} M. F. Atiyah, V. Patodi, and I. M. Singer, \emph{Spectral
asymmetry and Riemannian geometry.} \emph{III}, Proc. Camb. Phil. Soc.,
{\bfseries 79}(1976), 71--99.


%
%
%

\bibitem{BCPRSW}
M.-T. Benameur, A.~L. Carey, J.~Phillips, A.~Rennie, F.~A. Sukochev, and K.~P.
  Wojciechowski,
\newblock An analytic approach to spectral flow in von {N}eumann algebras.
\newblock In {\em Analysis, geometry and topology of elliptic operators}, pages
  297--352. World Sci. Publ., Hackensack, NJ, 2006.
\bibitem{BeF} M. T. Benameur and T. Fack \emph{Type II noncommutative geometry,
I. Dixmier trace in von Neumann algebras},  Adv. Math. {\bf 199} (2006), 29--87.
%






\bibitem{B1} M. Breuer, \emph{Fredholm theories in von Neumann
algebras. I}, Math. Ann., {\bfseries 178}(1968), 243-254.

\bibitem{B2} M. Breuer, \emph{Fredholm theories in von Neumann
algebras. II}, Math. Ann., {\bfseries 180}(1969), 313-325.


\bibitem{BrK} L. G. Brown and H. Kosaki,
\emph{Jensen's inequality in semifinite von Neumann algebras}, J.
Operator Theory, {\bf 23} (1990), 3--19.
%
%
%

\bibitem{CGRS2} A. L. Carey, V. Gayral, A. Rennie, F. Sukochev, {\em Index theory for 
locally compact noncommutative geometries}, arXiv:1107:0805.



\bibitem{CP1} A. L. Carey and J. Phillips, \emph{Unbounded Fredholm
modules and spectral flow}, Canad J. Math. {\bf vol. 50} (1998),
673--718.

\bibitem{CP2} A. L. Carey and J. Phillips,
{\em Spectral flow in $\Theta$-summable
Fredholm modules, eta invariants and the JLO cocycle}, J. $K$-Theory
{\bf 31} (2004), 135--194.


\bibitem{CPRS2} A. L. Carey, J. Phillips, A. Rennie, F. A. Sukochev,
\emph{The local index formula in semifinite von Neumann algebras
I. Spectral flow}, Adv. Math. {\bf 202} No. 2 (2006),  451--516.

\bibitem{CPRS3} A. L. Carey, J. Phillips, A. Rennie, F. A. Sukochev, \emph{
The local index formula in semifinite von Neumann algebras II: the even case},
Adv. Math. {\bf 202} No. 2 (2006),  517--554.


%
\bibitem{CPS2} A. L. Carey, J. Phillips, F. A. Sukochev,
\emph{Spectral flow and Dixmier traces}, Adv. Math. {\bf 173} (2003), 68--113.







\bibitem{Co1} A. Connes, \emph{Noncommutative differential geometry}, Publ.
Math. Inst. Hautes \'Etudes Sci., {\bf series 62} (1985), 41--44.
%






%
\bibitem{CM} A. Connes and H. Moscovici,
\emph{The local index formula in noncommutative geometry}, Geom. Funct. Anal.
 {\bf 5} (1995), 174--243.
%
%








\bibitem{Dix} J. Dixmier, \emph{Les alg\`ebres d'op\'erateurs dans l'espace
Hilbertien (Alg\`ebres de von Neumann)}, Gauthier-Villars, Paris, 1969.








%
%
\bibitem{Ge} M. Georgescu, PhD Thesis, University of Victoria, Canada, BC.

\bibitem{G} E. Getzler, \emph{The odd Chern character in cyclic homology
and spectral flow}, Topology, {\bf 32} (1993), 489-507.


%
%
%
%
%
%
%
%
%
%
%
\bibitem{Hi} N. Higson, \emph{The local index formula in noncommutative
geometry}, \emph{Contemporary Developments in Algebraic K-Theory,
ICTP Lecture Notes},
no. 15 (2003), 444--536.






%
%
%
%
%


\bibitem{L} M. Lesch,  \emph{On the index of the infinitesimal generator of
a flow}, J. Operator Theory {\bf 26} (1991), 73-92.


%










%
%


\bibitem{Ph1} J. Phillips, \emph{Self-adjoint Fredholm operators and
spectral flow}, Canad. Math. Bull., {\bf 39}(1996), 460--467.

\bibitem{Ph2} J. Phillips,
\emph{Spectral flow in type I and type II factors-a new approach},
Fields Institute Communications, {\bf vol. 17}(1997), 137--153.



\bibitem{PR}  John Phillips and Iain Raeburn \emph{An index theorem for
Toeplitz operators with noncommutative symbol space}, J. Funct. Anal.,
{\bf 120} (1993), 239--263.


%
%





%

%








\bibitem{Singer} I. M. Singer, \emph{Eigenvalues of the Laplacian and invariants
of manifolds}, Proceedings of the International Congress, Vancouver 1974,
{\bf vol. I}, 187-200.


%



%
%
%

\end{thebibliography}
\end{document}